\documentclass[12pt]{amsart}

\usepackage{amsmath}
\usepackage{amssymb}
\usepackage{latexsym}
\usepackage{graphicx} 
\usepackage{mathrsfs}
\usepackage{mathtools}
\usepackage{amsthm}
\usepackage{tikz}
\usetikzlibrary{cd}
\allowdisplaybreaks[4]

\usepackage[colorlinks,citecolor=blue,linkcolor=blue,linktocpage,unicode]{hyperref}

\numberwithin{equation}{section}
\numberwithin{figure}{section}

\newtheorem{thm}{Theorem}[section]
\newtheorem{prop}[thm]{Proposition}
\newtheorem{lemma}[thm]{Lemma}
\newtheorem{cor}[thm]{Corollary}

\theoremstyle{definition}
\newtheorem{example}[thm]{Example}
\newtheorem{defin}[thm]{Definition}
\newtheorem{remark}[thm]{Remark}
\newtheorem{assumption}[thm]{Assumption}

\title[Delzant type theorem for equivariant hypersurfaces]{Delzant type theorem for torus-equivariantly embedded toric hypersurfaces}
\author{Kentaro Yamaguchi}
\date{\today}

\address{Department of Mathematical Sciences, Tokyo Metropolitan University, 1-1 Minami-Ohsawa, Hachioji, Tokyo, 192-0397, Japan}
\email{c201596e@alm.icu.ac.jp}

\begin{document}

\begin{abstract}
In the previous work, we study the moment polytope of the closure of the complex subtorus orbit in a symplectic toric manifold associated to an affine subspace when the closure is a smooth complex submanifold. 
In this paper, we clarify the condition for nonsingularity of the closure of the codimension one complex subtorus orbit in terms of polytopes.
The main result is a generalization of Delzant theorem.
\end{abstract}

\maketitle


\section{Introduction}
\label{sec: introduction}

Symplectic toric manifolds are $2n$-dimensional symplectic manifolds equipped with an effective Hamiltonian action of an $n$-dimensional torus.
The image of the moment map for the Hamiltonian torus action is the convex hull of the image of the fixed points of the Hamiltonian action \cite{MR642416,MR664117}.
Due to the work of Delzant \cite{MR984900}, there is a one-to-one correspondence between symplectic toric manifolds and certain convex polytopes appeared as the moment polytopes of symplectic toric manifolds. The moment polytopes of symplectic toric manifolds are called \textit{Delzant polytopes}.
Moreover, symplectic toric manifolds canonically admit a K\"{a}hler metric, which is called the \textit{Guillemin metric} \cite{MR1293656,MR1301331,MR2039163}.
Using the data of a Delzant polytope, the complements of the toric divisors $D$ in the corresponding symplectic toric manifold $X$ can be identified with the complex torus $(\mathbb{C}^{\ast})^{n}$.

In \cite{MR4938007}, we construct the closure $\overline{C(V)} \subset X$ of the $(\mathbb{C}^{\ast})^{k}$-orbit $C(V) \subset X \setminus D \cong (\mathbb{C}^{\ast})^{n}$ from the data of a $k$-dimensional affine subspace $V \subset \mathfrak{t}^{n} \cong \mathbb{R}^{n}$.
The construction of $\overline{C(V)}$ is motivated by Yamamoto's construction \cite{MR3856842} of the Leung--Yau--Zaslow correspondence \cite{MR1894858} in the semi-flat case of Strominger--Yau--Zaslow mirror symmetry \cite{MR1429831}.
The closure $\overline{C(V)}$ can be expressed by the zero locus of polynomials $f^{\lambda}_{k+1},\ldots,f^{\lambda}_{n}$ for each vertex $\lambda$ of the Delzant polytope, which are determined by the data of the Delzant polytope and the affine subspace.
Note that the closure $\overline{C(V)}$ might have singularity at the intersection with the toric divisors $D$.
We show in \cite[Theorem 4.20]{MR4938007} that if the closure $\overline{C(V)}$ is a smooth $k$-dimensional complex submanifold in the symplectic toric manifold $X$, then the moment polytope for the Hamiltonian $k$-dimensional torus action on $\overline{C(V)}$ coincides with the moment polytope for the Hamiltonian $k$-dimensional subtorus action on $X$.
Moreover, the submanifold $\overline{C(V)}$ is a symplectic toric manifold with respect to the $k$-dimensional torus action on $\overline{C(V)}$.
In particular, if $\overline{C(V)}$ is a complex submanifold, then the moment polytope for the Hamiltonian $k$-dimensional subtorus action on $X$ is a Delzant polytope. 
We call the complex submanifold $\overline{C(V)}$ the \textit{torus-equivariantly embedded toric manifold}.

Even though the moment polytope for the Hamiltonian $k$-dimensional subtorus action on $X$ is a Delzant polytope as a convex polytope for some affine subspace $V$, $\overline{C(V)}$ might not be a smooth complex submanifold.
For example, let $X = \mathbb{C}P^{2}$.
For any $1$-dimensional affine subspace $V$ in $\mathfrak{t}^{2} \cong \mathbb{R}^{2}$, the moment polytope for the Hamiltonian one-dimensional subtorus action associated with $V$ becomes an interval in $(\mathfrak{t}^{1})^{\ast} \cong \mathbb{R}$, which can be seen as the Delzant polytope of $\mathbb{C}P^{1}$. However, there are many affine subspaces $V$ such that the closure $\overline{C(V)}$ has a singularity (see \cite[Section 5.1]{MR4938007}).

Under the Delzant correspondence, we expect that objects in symplectic toric geometry shall be characterized combinatorially in terms of polytopes. The main result in this paper is to give a characterization of the nonexistence of singular points in $\overline{C(V)}$ in terms of (moment) polytopes. 
Moreover, through the Leung--Yau--Zaslow correspondence, we expect that this characterization must have a relation with an asymptotic condition of conormal bundles of $V$ in $\mathfrak{t}^{n}$ (in the sense of \cite{MR2564372,MR3948684} for example).

\subsection{Main Results}
\label{subsec: main theorem}

In this paper, we show that a Delzant type correspondence for torus-equivariantly embedded toric manifolds of codimension one, i.e. $k = n-1$.

Let $V$ be an $(n-1)$-dimensional affine subspace in $\mathfrak{t}^{n} \cong \mathbb{R}^{n}$ with rational slope, i.e. $\vec{V} \cap \mathbb{Z}^{n} \cong \mathbb{Z}^{n-1}$, where $\vec{V}$ is the linear part of $V$.
Then, we may take a primitive $\mathbb{Z}$-basis $p_{1},\ldots, p_{n-1} \in \mathbb{Z}^{n}$ of $\vec{V} \cap \mathbb{Z}^{n}$ and $a \in \mathbb{R}^{n}$ such that $V = \mathbb{R}p_{1} + \cdots + \mathbb{R}p_{n-1} + a$.
We define the injective homomorphism $i_{V}:T^{n-1} \to T^{n}$ by 
\begin{equation*}
    i_{V}(t_{1},\ldots, t_{n-1})
    :=
    \left(
        \prod_{l=1}^{n-1}t_{l}^{\langle p_{l},e_{1} \rangle}, \ldots, \prod_{l=1}^{n-1}t_{l}^{\langle p_{l},e_{n} \rangle}
    \right),
\end{equation*}
where $e_{1},\ldots,e_{n}$ are the standard basis of $(\mathfrak{t}^{n})^{\ast} \cong \mathbb{R}^{n}$.
We introduce some notions concerning the pair of a Delzant polytope and an affine subspace.

\begin{defin}
\label{def: vertex and no direction vectors are perp}
Let $\Delta$ be a Delzant polytope in $(\mathfrak{t}^{n})^{\ast}$ and $V = \mathbb{R}p_{1} + \cdots + \mathbb{R}p_{n-1} + a$ an $(n-1)$-dimensional affine subspace in $\mathfrak{t}^{n}$ with rational slope. 
A vertex $\lambda$ of the polytope $\Delta$ is a \textit{good vertex with respect to the pullback $i_{V}^{\ast}: (\mathfrak{t}^{n})^{\ast} \to (\mathfrak{t}^{n-1})^{\ast}$} if the vertex $\lambda$ satisfies the two conditions:
\begin{enumerate}
    \item $i_{V}^{\ast}(\lambda)$ is a vertex of the convex polytope $i_{V}^{\ast}(\Delta)$,
    \item for the direction vectors $v^{\lambda}_{1},\ldots, v^{\lambda}_{n}$ of the edges from the vertex $\lambda$ of the polytope $\Delta$, the vectors $i_{V}^{\ast}(v^{\lambda}_{1}),\ldots, i_{V}^{\ast}(v^{\lambda}_{n})$ in $(\mathfrak{t}^{n-1})^{\ast}$ are all nonzero.
\end{enumerate}
\end{defin}

\begin{defin}
\label{def: almost delzant}
Let $\Delta$ be a Delzant polytope in $(\mathfrak{t}^{n})^{\ast}$. An affine subspace $V$ in $\mathfrak{t}^{n}$ with rational slope is \textit{admissible with respect to the Delzant polytope $\Delta$} if any good vertex $\lambda$ with respect to the map $i_{V}^{\ast}$ satisfies the two conditions:
\begin{itemize}
    \item we can choose $n-1$ direction vectors $v^{\lambda}_{1},\ldots, v^{\lambda}_{j-1},v^{\lambda}_{j+1},\ldots, v^{\lambda}_{n}$ such that 
    $i_{V}^{\ast}(v^{\lambda}_{1}),\ldots, i_{V}^{\ast}(v^{\lambda}_{j-1}),i_{V}^{\ast}(v^{\lambda}_{j+1}),\ldots,i_{V}^{\ast}(v^{\lambda}_{n})$ are linearly independent and $i_{V}^{\ast}(\{\sum_{i=1}^{n}a_{i}v^{\lambda}_{i}\mid a_{i} \geq 0 \}) = \{\sum_{i\neq j}b_{i}i_{V}^{\ast}(v^{\lambda}_{i}) \mid b_{i} \geq 0 \}$,
    \item the vectors $i_{V}^{\ast}(v^{\lambda}_{1}),\ldots, i_{V}^{\ast}(v^{\lambda}_{j-1}),i_{V}^{\ast}(v^{\lambda}_{j+1}),\ldots,i_{V}^{\ast}(v^{\lambda}_{n}) \in \mathbb{Z}^{n-1}$ from the vertex $i_{V}^{\ast}(\lambda)$ of $i_{V}^{\ast}(\Delta)$ form a $\mathbb{Z}$-basis of $\mathbb{Z}^{n-1}$.
\end{itemize}
\end{defin}

Note that the first condition in Definition \ref{def: almost delzant} are equivalent to the condition that the polytope $i_{V}^{\ast}(\Delta)$ is simple.
In Section \ref{sec: example} we demonstrate examples of subspaces that are admissible and not admissible with respect to some Delzant polytopes.

The closure $\overline{C(V)}$ can be expressed by the zero locus of a polynomial $f^{\lambda}$ for each vertex $\lambda$ of $\Delta$.
Our first main result is as follows:
\begin{thm}[Theorem \ref{thm: almost delzant to torus equivariant submanifold}]
\label{thm: main}
Let $\Delta$ be a Delzant polytope in $(\mathfrak{t}^{n})^{\ast}$, $V = \mathbb{R}p_{1} + \cdots + \mathbb{R}p_{n-1} + a$ an $(n-1)$-dimensional affine subspace in $\mathfrak{t}^{n}$ with rational slope, $X$ a toric manifold associated with the Delzant polytope $\Delta$. 
If $V$ is admissible with respect to $\Delta$, then the rank of the Jacobian matrix $Df^{\lambda}$ is equal to one at any point of the zero locus of $f^{\lambda}$ for any vertex $\lambda$ of the Delzant polytope $\Delta$. In particular, $\overline{C(V)}$ is a smooth complex hypersurface in the toric manifold $X$.
\end{thm}

Since complex hypersurfaces $\overline{C(V)}$ are symplectic toric manifolds with respect to the Hamiltonian $T^{n-1}$-action, we obtain the following:

\begin{cor}
\label{cor: good then delzant}
If $V$ is admissible with respect to the Delzant polytope $\Delta$ of the toric manifold $X$, then the convex polytope $i_{V}^{\ast}(\Delta)$ is a Delzant polytope of the complex hypersurface $\overline{C(V)}$ in $X$.
\end{cor}

We also show the converse of Theorem \ref{thm: main}.

\begin{thm}[Theorem \ref{thm: torus equivariant submanifold to almost delzant}]
\label{thm: main2}
Let $V = \mathbb{R}p_{1} + \cdots + \mathbb{R}p_{n-1} + a$ be an $(n-1)$-dimensional affine subspace in $\mathfrak{t}^{n}$ with rational slope.
If the rank of the Jacobian matrix $Df^{\lambda}$ is equal to one at any point of the zero locus of $f^{\lambda}$ for any vertex $\lambda$ of the Delzant polytope $\Delta$ of the toric manifold $X$, then $V$ is admissible with respect to $\Delta$.
\end{thm}

From Theorem \ref{thm: main} and Theorem \ref{thm: main2}, we obtain the following:

\begin{cor}
\label{cor: main}
Let $\Delta$ be a Delzant polytope of a symplectic toric manifold $X$.
Then, $\overline{C(V)}$ is a torus-equivariantly embedded toric hypersurface in $X$ 
if and only if 
$\Delta$ is a good Delzant polytope with respect to the map $i_{V}^{\ast}$.
\end{cor}

Corollary \ref{cor: main} characterizes the condition when $\overline{C(V)}$ is a smooth complex hypersurface in terms of the combinatorics of Delzant polytopes.

\subsection{Sketch of the Proofs}
We explain the sketch of the proof of Theorem \ref{thm: main}.
First, in Section \ref{subsec: case of orthogonal} we show that if the vertex $\lambda$ of the Delzant polytope $\Delta$ does not satisfy the condition (2) in Definition \ref{def: vertex and no direction vectors are perp}, then the rank of the Jacobian matrix $Df^{\lambda}$ is equal to one.
Next, in Section \ref{subsec: case of not vertex} we show that if the vertex $\lambda$ satisfies the condition (2) but does not satisfy the condition (1) in Definition \ref{def: vertex and no direction vectors are perp}, then the rank of the Jacobian matrix $Df^{\lambda}$ is equal to one. 
Finally, in Section \ref{subsec: case of vertex} we consider the case when the vertex $\lambda$ is good with respect to the map $i_{V}^{\ast}$.
The detailed proof of Theorem \ref{thm: main} is given in Section \ref{subsec: proof}.

The proof of Theorem \ref{thm: main2} is based on the observation of the conditions of the sets $\mathcal{I}^{+}_{\lambda},\mathcal{I}^{-}_{\lambda}$, which are defined in Definition \ref{def: closure}.

\subsection{Outline}
This paper is organized as follows.
In Section \ref{sec: review}, we prepare for something about torus-equivariantly embedded toric hypersurfaces.
In Section \ref{sec: delzant of equiv toric}, we show Theorem \ref{thm: main}.
In Section \ref{sec: torus-equivariant to delzant}, we show Theorem \ref{thm: main2}.
In Section \ref{sec: example}, we demonstrate some examples to illustrate admissible subspaces.

\subsection*{Acknowlegements}
The author is grateful to the advisor, Manabu Akaho, for a lot of encouragements and supports.
The author would also like to thank Yuichi Kuno and Yasuhito Nakajima for helpful discussions, and would also like to thank the anonymous referee for helpful comments which improve the exposition. 
This work was supported by JST SPRING, Grant Number JPMJSP2156.

\section{Review of Torus-equivariantly Embedded Toric Manifolds}
\label{sec: review}

In this section, we review the settings in \cite{MR4938007}.
We also show some properties of the rank of the Jacobian matrix $Df^{\lambda}$ used in this paper.

\subsection{Symplectic Toric Manifolds} 
\label{subsec: toric manifolds}

In \cite{MR984900}, symplectic toric manifolds are completely classified by the certain convex polytopes, known as \textit{Delzant polytopes}:

\begin{defin}
\label{def: delzant polytopes}
A convex polytope $\Delta$ in $(\mathfrak{t}^{n})^{\ast} \cong \mathbb{R}^{n}$ is \textit{Delzant} if $\Delta$ satisfies the following three properties:
\begin{itemize}
    \item simple; each vertex has $n$ edges,
    \item rational; the direction vectors $v^{\lambda}_{1},\ldots, v^{\lambda}_{n}$ of the edges from any vertex $\lambda \in \Lambda$ can be chosen as integral vectors,
    \item smooth; the vectors $v^{\lambda}_{1},\ldots, v^{\lambda}_{n} \in \mathbb{Z}^{n}$ chosen as above form a $\mathbb{Z}$-basis of $\mathbb{Z}^{n}$.
\end{itemize}
Here, $\Lambda$ denotes the set of the vertices of $\Delta$.
\end{defin}

Some literatures (for example, \cite{MR2039163}) define Delzant polytopes in terms of the inward pointing normal vectors $u^{\lambda}_{1},\ldots, u^{\lambda}_{n}$ to the $n$ facets sharing each vertex instead of the direction vectors from each vertex.
In fact, these two ways to define Delzant polytopes are equivalent because of the following equation:
\begin{equation*}
    {^t [u^{\lambda}_{1} \; \cdots \; u^{\lambda}_{n} ]}[v^{\lambda}_{1} \; \cdots \; v^{\lambda}_{n}] = E_{n}
\end{equation*}
holds for any vertex $\lambda$ of $\Delta$ (see \cite[Lemma 3.10]{MR4938007} for example).

Due to the Delzant construction, we obtain the corresponding symplectic toric manifold $X$ from the data of a given Delzant polytope $\Delta$.
Here we give the expression of a system of the inhomogeneous coordinate charts on the corresponding symplectic toric manifold $X$ from the data of the Delzant polytope $\Delta$.
We define two matrices $Q^{\lambda} (= [Q^{\lambda}_{ij}])$ and $D^{\lambda\sigma} (= [d^{\lambda\sigma}_{ij}])$ by 
\begin{equation*}
    Q^{\lambda}
    =
    \left[
            v^{\lambda}_{1} \; \cdots \; v^{\lambda}_{n}
    \right], \;
    D^{\lambda\sigma} 
    =
    (Q^{\lambda})^{-1}Q^{\sigma}
\end{equation*}
for any vertices $\lambda,\sigma \in \Lambda$.

\begin{lemma}
From the data of a given Delzant polytope $\Delta$, we can construct an open covering $\{U_{\lambda}\}_{\lambda \in \Lambda}$ of the corresponding toric manifold $X$ and a set of maps $\{\varphi_{\lambda}:U_{\lambda}\to \mathbb{C}^{n}\}_{\lambda \in \Lambda}$ such that 
\begin{equation*}
    \varphi_{\sigma}\circ \varphi_{\lambda}^{-1} (z^{\lambda}) 
    = \left(
        \prod_{j=1}^{n}(z^{\lambda}_{j})^{d^{\lambda\sigma}_{j1}}, \ldots, \prod_{j=1}^{n}(z^{\lambda}_{j})^{d^{\lambda\sigma}_{jn}}
    \right)
\end{equation*}
for any vertices $\lambda,\sigma \in \Lambda$ such that $U_{\lambda}\cap U_{\sigma} \neq \emptyset$.
\end{lemma}

The detailed construction in terms of the notation used here is given in \cite[Section 2.1 and Section 3.2]{MR4938007}.
Moreover, the complement of toric divisors in a $2n$-dimensional symplectic toric manifold can be identified with a complex $n$-dimensional torus.

\subsection{Torus-equivariantly Embedded Toric Hypersurfaces}
\label{subsec: torus-equivariant}

Let $p_{1},\ldots,p_{n-1} \in \mathbb{Z}^{n}$ be primitive vectors which are linearly independent, and $a \in \mathbb{R}^{n}$.
We consider the $(n-1)$-dimensional affine subspace $V = \mathbb{R}p_{1} + \cdots + \mathbb{R}p_{n-1} + a$ in $\mathfrak{t}^{n} \cong \mathbb{R}^{n}$.
Write $\vec{V}$ as the linear part of $V$. 
Assume that $V$ has a rational slope, i.e. we may assume that $p_{1},\ldots, p_{n-1} \in \mathbb{Z}^{n}$ form a $\mathbb{Z}$-basis of $\vec{V} \cap \mathbb{Z}^{n} \cong \mathbb{Z}^{n-1}$.
Let $q \in \mathbb{Z}^{n}$ be a primitive basis of the orthogonal subspace to $\vec{V}$ in $\mathfrak{t}^{n} \cong \mathbb{R}^{n}$.

Through the log-affine coordinate $(\mathbb{C}^{\ast})^{n} \cong \mathfrak{t}^{n} \times T^{n}$, we define the complex subtorus $C(V) \cong V \times T^{n-1}$ (see \cite[Proposition 4.2]{MR4938007} for detail).
Recall that the complement of toric divisors in a toric manifold $X$ can be identified with a complex $n$-dimensional torus $(\mathbb{C}^{\ast})^{n}$.
The closure $\overline{C(V)}$ of $C(V)$ in the toric manifold $X$ is expressed as follows:

\begin{defin}
\label{def: closure}
The closure $\overline{C(V)}$ of $C(V)$ in the toric manifold $X$ is defined as $\overline{C(V)} = \bigcup_{\lambda \in \Lambda} \varphi_{\lambda}^{-1}(\overline{C_{\lambda}(V)})$, where $\overline{C_{\lambda}(V)} \subset \varphi_{\lambda}(U_{\lambda}) \cong \mathbb{C}^{n}$ is the zero locus of $f^{\lambda}$. 
Here, $f^{\lambda}$ are defined by 
\begin{equation*}
    f^{\lambda}(z^{\lambda})
    =
    \prod_{j \in \mathcal{I}^{+}_{\lambda}\setminus \mathcal{I}^{0}_{\lambda}}(z^{\lambda}_{j})^{\langle u^{\lambda}_{j}, q \rangle}
    -
    e^{\langle a, q \rangle}
    \prod_{j \in \mathcal{I}^{-}_{\lambda}\setminus \mathcal{I}^{0}_{\lambda}}(z^{\lambda}_{j})^{-\langle u^{\lambda}_{j}, q \rangle},
\end{equation*}
where
\begin{align*}
    &\mathcal{I}^{+}_{\lambda}
    = \{i \in \{1,\ldots, n\} \mid \langle u^{\lambda}_{i},q \rangle \geq 0 \}, \\
    &\mathcal{I}^{-}_{\lambda}
    = \{i \in \{1,\ldots, n\} \mid \langle u^{\lambda}_{i},q \rangle \leq 0 \},\\
    &\mathcal{I}^{0}_{\lambda}
    = \{i \in \{1,\ldots, n\} \mid \langle u^{\lambda}_{i},q \rangle = 0 \}.
\end{align*}
\end{defin}
Note that if $\mathcal{I}^{+}_{\lambda} \setminus \mathcal{I}^{0}_{\lambda} = \emptyset$, then $\prod_{j \in \mathcal{I}^{+}_{\lambda} \setminus \mathcal{I}^{0}_{\lambda}} (z^{\lambda}_{j})^{\langle u^{\lambda}_{j},q \rangle} = 1$.
Similarly, if $\mathcal{I}^{-}_{\lambda} \setminus \mathcal{I}^{0}_{\lambda} = \emptyset$, then $\prod_{j \in \mathcal{I}^{-}_{\lambda} \setminus \mathcal{I}^{0}_{\lambda}} (z^{\lambda}_{j})^{-\langle u^{\lambda}_{j},q \rangle} = 1$.

Also, $\overline{C(V)}$ can be reprenseted as the zero locus of the polynomial $f$ of $X$ with $f|_{\varphi_{\lambda}(U_{\lambda})} = f^{\lambda}$.
We define a singular point of $\overline{C(V)}$ in $X$ by the point where the Jacobian matrix $Df^{\lambda}$ of $f^{\lambda}$ becomes zero for some $\lambda$. In this case, we say that $\overline{C(V)}$ has a singularity.
Similarly, we say that $\overline{C(V)}$ is nonsingular if $\overline{C(V)}$ has no singular point.

\begin{remark}
The rank of the Jacobian matrix $Df^{\lambda}$ of $f^{\lambda}$ might be equal to zero at some points in the toric divisors.
In particular, $\overline{C(V)}$ might have a singularity at some points in the toric divisors.
\end{remark}

Note that if $\overline{C(V)}$ is a smooth complex submanifold in the toric manifold $X$, then $\overline{C(V)}$ has the effective Hamiltonian $T^{n-1}$-action through the map $i_{V}:T^{n-1} \to T^{n}$ and the $T^{n}$-action on $X$ (see \cite{MR4938007}).

\begin{prop}\label{prop: toric}
If $\overline{C(V)}$ is a complex submanifold in the toric manifold $X$, then $\overline{C(V)}$ is toric with respect to the $T^{n-1}$-action on $\overline{C(V)}$. 
\end{prop}

The vectors $q, v^{\lambda}_{1},\ldots, v^{\lambda}_{n}, u^{\lambda}_{1},\ldots, u^{\lambda}_{n}$ used here have the following relation:

\begin{lemma}[{\cite[Lemma 4.14]{MR4938007}}]
\label{lemma: u,q,p,v}
\begin{equation*}
    \sum_{j=1}^{n}\langle u^{\lambda}_{j},q \rangle i_{V}^{\ast}(v^{\lambda}_{j})
    = 0
\end{equation*}
holds.
\end{lemma}

In this case, the rank of the Jacobian matrix $Df^{\lambda}$ is independent of the choice of $a \in \mathbb{R}^{n}$ of the affine subspace $V = \mathbb{R}p_{1} + \cdots + \mathbb{R}p_{n-1} + a$ \cite[Proposition 2.12]{yamaguchi2024delzant}, i.e. we may assume that $a = 0$.

\subsection{Properties of the Rank of the Jacobian Matrices}
\label{subsec: rank of torus-equivariant}

In this section, we prepare for the proof of the main results.

\begin{lemma}
\label{lemma: sign only leads to nonsingular}
Let $\lambda$ be a vertex of $\Delta$.
\begin{itemize}
    \item If $\mathcal{I}^{+}_{\lambda} \setminus \mathcal{I}^{0}_{\lambda} = \emptyset$, then $\overline{C_{\lambda}(V)} \cap \{z^{\lambda}_{j} = 0\} = \emptyset$ for any $j \in \mathcal{I}^{-}_{\lambda} \setminus \mathcal{I}^{0}_{\lambda}$.
    \item If $\mathcal{I}^{-}_{\lambda} \setminus \mathcal{I}^{0}_{\lambda} = \emptyset$, then $\overline{C_{\lambda}(V)} \cap \{z^{\lambda}_{j} = 0\} = \emptyset$ for any $j \in \mathcal{I}^{+}_{\lambda} \setminus \mathcal{I}^{0}_{\lambda}$.
\end{itemize}
\end{lemma}
\begin{proof}
We assume that $\mathcal{I}^{+}_{\lambda} \setminus \mathcal{I}^{0}_{\lambda} = \emptyset$. 
Then, since the defining equation $f^{\lambda}(z^{\lambda}) = 0$ for $\overline{C_{\lambda}(V)}$ is 
\begin{equation*}
    f^{\lambda}(z^{\lambda})
    =
    1 - \prod_{j \in \mathcal{I}^{-}_{\lambda}\setminus \mathcal{I}^{0}_{\lambda}} (z^{\lambda}_{j})^{-\langle u^{\lambda}_{j}, q \rangle},
\end{equation*}
we obtain that $\overline{C_{\lambda}(V)} \cap \{z^{\lambda}_{j} = 0\} = \emptyset$ for any $j \in \mathcal{I}^{-}_{\lambda} \setminus \mathcal{I}^{0}_{\lambda}$.

Similarly, if we assume that $\mathcal{I}^{-}_{\lambda} \setminus \mathcal{I}^{0}_{\lambda} = \emptyset$, then we obtain $\overline{C_{\lambda}(V)} \cap \{z^{\lambda}_{j} = 0\} = \emptyset$ for any $j \in \mathcal{I}^{+}_{\lambda} \setminus \mathcal{I}^{0}_{\lambda}$.
\end{proof}

\begin{lemma}
\label{lemma: rank of zero locus only sign}
If $\mathcal{I}^{+}_{\lambda} \setminus \mathcal{I}^{0}_{\lambda} = \emptyset$ (or $\mathcal{I}^{-}_{\lambda} \setminus \mathcal{I}^{0}_{\lambda} = \emptyset$), then $\overline{C_{\lambda}(V)}$ is nonsingular.
\end{lemma}
\begin{proof}
Assume that $\mathcal{I}^{+}_{\lambda} \setminus \mathcal{I}^{0}_{\lambda} = \emptyset$. From the proof of Lemma \ref{lemma: sign only leads to nonsingular}, the defining equation $f^{\lambda}(z^{\lambda}) = 0$ for $\overline{C_{\lambda}(V)}$ is 
\begin{equation*}
    f^{\lambda}(z^{\lambda})
    =
    1 - \prod_{j \in \mathcal{I}^{-}_{\lambda}\setminus \mathcal{I}^{0}_{\lambda}} (z^{\lambda}_{j})^{-\langle u^{\lambda}_{j}, q \rangle}.
\end{equation*}
Since we obtain 
\begin{equation*}
    \frac{\partial f^{\lambda}}{\partial z^{\lambda}_{i}}
    =
    \begin{cases*}
        \langle u^{\lambda}_{i},q\rangle  \prod_{j \in \mathcal{I}^{-}_{\lambda}\setminus \mathcal{I}^{0}_{\lambda}} (z^{\lambda}_{j})^{-\langle u^{\lambda}_{j},q \rangle- \delta_{ij}} & if $i \in \mathcal{I}^{-}_{\lambda}\setminus \mathcal{I}^{0}_{\lambda}$, \\
        0 & otherwise,
    \end{cases*}
\end{equation*}
the points where the rank of the Jacobian matrix $Df^{\lambda}$ of $f^{\lambda}$ is equal to zero should be in $\overline{C_{\lambda}(V)} \cap (\bigcup_{j \in \mathcal{I}^{-}_{\lambda}\setminus\mathcal{I}^{0}_{\lambda}}\{z^{\lambda}_{j}=0 \})$.
However, since $\overline{C_{\lambda}(V)} \cap (\bigcup_{j \in \mathcal{I}^{-}_{\lambda}\setminus\mathcal{I}^{0}_{\lambda}}\{z^{\lambda}_{j}=0 \}) = \emptyset$ from Lemma \ref{lemma: sign only leads to nonsingular}, the rank of $Df^{\lambda}$ is equal to one, i.e. $\overline{C_{\lambda}(V)}$ does not have a singular point.

Similarly, we assume that $\mathcal{I}^{-}_{\lambda} \setminus \mathcal{I}^{0}_{\lambda} = \emptyset$, then we obtain that the rank of $Df^{\lambda}$ is equal to one.
\end{proof}

\begin{lemma}
\label{lemma: splitting}
Assume that 
\begin{equation*}
    \mathcal{I}^{+}_{\lambda} = \{1,\ldots,n-1\}, \; \mathcal{I}^{-}_{\lambda} \setminus \mathcal{I}^{0}_{\lambda} = \{n\}
\end{equation*}
or
\begin{equation*}
    \mathcal{I}^{+}_{\lambda} \setminus \mathcal{I}^{0}_{\lambda}= \{n\}, \; \mathcal{I}^{-}_{\lambda} = \{1,\ldots,n-1\}.
\end{equation*}
If $\langle u^{\lambda}_{n},q \rangle = \pm 1$, then $\overline{C_{\lambda}(V)}$ is nonsingular.
\end{lemma}
\begin{proof}
Assume that $\mathcal{I}^{+}_{\lambda} = \{1,\ldots,n-1\}, \; \mathcal{I}^{-}_{\lambda} \setminus \mathcal{I}^{0}_{\lambda} = \{n\}$.
From the defining equation for $\overline{C_{\lambda}(V)}$, we obtain that 
\begin{equation*}
    \frac{\partial f^{\lambda}}{\partial z^{\lambda}_{i}} = 
    \begin{cases*}
        \langle u^{\lambda}_{i},q\rangle  \prod_{j \in \mathcal{I}^{+}_{\lambda}\setminus \mathcal{I}^{0}_{\lambda}} (z^{\lambda}_{j})^{\langle u^{\lambda}_{j},q \rangle- \delta_{ij}} & if $i =1,\ldots,n-1$, \\
        \langle u^{\lambda}_{n},q \rangle (z^{\lambda}_{n})^{-\langle u^{\lambda}_{n},q \rangle - 1} & if $i=n$.
    \end{cases*}
\end{equation*}
If $\langle u^{\lambda}_{n},q \rangle = -1$, then 
\begin{equation*}
    \frac{\partial f^{\lambda}}{\partial z^{\lambda}_{n}} = -1.
\end{equation*}
Therefore, the rank of the Jacobian matrix $Df^{\lambda}$ is equal to one, i.e. $\overline{C_{\lambda}(V)}$ is nonsingular.
\end{proof}

\section{Delzant Polytopes of Torus-equivariantly Embedded Toric Hypersurfaces}
\label{sec: delzant of equiv toric}

In this section, we show the first main result (Theorem \ref{thm: almost delzant to torus equivariant submanifold}).

Since $i_{V}^{\ast}:(\mathfrak{t}^{n})^{\ast} \to (\mathfrak{t}^{n-1})^{\ast}$ is surjective, we have the following:
\begin{lemma}
\label{lemma: basis to basis via linear map}
Let $v_{1},\ldots, v_{n} \in \mathbb{R}^{n} \cong (\mathfrak{t}^{n})^{\ast}$ be linearly independent.
There exist $j_{1},\ldots,j_{n-1} \in \{1,\ldots, n\}$ such that $i_{V}^{\ast}(v_{j_{1}}),\ldots, i_{V}^{\ast}(v_{j_{n-1}}) \in (\mathfrak{t}^{n-1})^{\ast} \cong \mathbb{R}^{n-1}$ are linearly independent.
\end{lemma}

Hereafter, we assume that $j_{1} = 1,\ldots, j_{n-1} = n-1$ otherwise specified.
Let $\lambda$ be a vertex of the Delzant polytope $\Delta$. Define $\mathcal{J}_{\lambda} = \{i \mid \langle p_{1},v^{\lambda}_{i} \rangle = \cdots = \langle p_{n-1},v^{\lambda}_{i} \rangle = 0 \}$.

\subsection{Case of $\mathcal{J}_{\lambda} \neq \emptyset$}
\label{subsec: case of orthogonal}

In this section, we show that if the vertex $\lambda$ of $\Delta$ does not satisfy the condition (2) in Definition \ref{def: vertex and no direction vectors are perp}, then $\overline{C_{\lambda}(V)}$ is nonsingular.
\begin{lemma}
\label{lemma: not condition 2 iff condition}
The vertex $\lambda$ of $\Delta$ does not satisfy the condition (2) in Definition \ref{def: vertex and no direction vectors are perp} if and only if $\mathcal{J}_{\lambda} \neq \emptyset$.
\end{lemma}
\begin{proof}
If the vertex $\lambda$ of $\Delta$ does not satisfy the condition (2) in Definition \ref{def: vertex and no direction vectors are perp}, then there exists a direction vector $v^{\lambda}_{j}$ such that the vector $i_{V}^{\ast}(v^{\lambda}_{j}) \in (\mathfrak{t}^{n-1})^{\ast}$ is zero.
Note that the pullback $i_{V}^{\ast}:(\mathfrak{t}^{n})^{\ast} \to (\mathfrak{t}^{n-1})^{\ast}$ is given by 
\begin{equation*}
    i_{V}^{\ast}(\xi)
    =
    \left(
    \langle p_{1},\xi \rangle, \ldots, \langle p_{n-1},\xi \rangle
    \right).
\end{equation*}
Since $i_{V}^{\ast}(v^{\lambda}_{j}) = (0,\ldots, 0)$, we obtain $\langle p_{1},v^{\lambda}_{j} \rangle = 0,\ldots, \langle p_{n-1},v^{\lambda}_{j} \rangle = 0$, i.e. $j \in \mathcal{J}_{\lambda}$. In particular, $\mathcal{J}_{\lambda} \neq \emptyset$.

If $\mathcal{J}_{\lambda} \neq \emptyset$, then there exists an element $j_{0} \in \mathcal{J}_{\lambda}$.
By the definition of the set $\mathcal{J}_{\lambda}$, we have $\langle p_{1},v^{\lambda}_{j_{0}} \rangle = 0,\ldots, \langle p_{n-1},v^{\lambda}_{j_{0}} \rangle = 0$, i.e. the vector $i_{V}^{\ast}(v^{\lambda}_{j_{0}})$ is zero.
\end{proof}
In particular, $\mathcal{J}_{\lambda} \neq \emptyset$ is equivalent to $\vert \mathcal{J}_{\lambda}\vert = 1$.

\begin{lemma}
\label{lemma: vertical edge leads to nonsingular}
If $\mathcal{J}_{\lambda} \neq \emptyset$, then $\mathcal{I}^{+}_{\lambda}\setminus \mathcal{I}^{0}_{\lambda} = \emptyset$ or $\mathcal{I}^{-}_{\lambda}\setminus \mathcal{I}^{0}_{\lambda} = \emptyset$.
\end{lemma}
\begin{proof}
Let $j_{0} \in \mathcal{J}_{\lambda}$.
By the definition of $\mathcal{J}_{\lambda}$, we obtain $v^{\lambda}_{j_{0}} \in V^{\perp} = \mathbb{R}q$. 
Since the vectors $v^{\lambda}_{j_{0}},q$ are nonzero, there exists $b \in \mathbb{R}\setminus \{0\}$ such that $q = bv^{\lambda}_{j_{0}}$.
Since $\langle u^{\lambda}_{i},v^{\lambda}_{j} \rangle = \delta_{ij}$ (see \cite[Lemma 3.10]{MR4938007}), we have 
\begin{equation*}
    \langle u^{\lambda}_{i},q \rangle 
    =
    \langle u^{\lambda}_{i},bv^{\lambda}_{j_{0}} \rangle 
    =
    b \langle u^{\lambda}_{i},v^{\lambda}_{j_{0}} \rangle 
    =
    b \delta_{ij_{0}}
\end{equation*}
for any $i = 1,\ldots, n$.
If $b > 0$, then $i \in \mathcal{I}^{+}_{\lambda}$ for any $i$, i.e. $\mathcal{I}^{-}_{\lambda}\setminus \mathcal{I}^{0}_{\lambda} = \emptyset$.
If $b < 0$, then $i \in \mathcal{I}^{-}_{\lambda}$ for any $i$, i.e. $\mathcal{I}^{+}_{\lambda}\setminus \mathcal{I}^{0}_{\lambda} = \emptyset$.
\end{proof}

From Lemma \ref{lemma: vertical edge leads to nonsingular} and Lemma \ref{lemma: rank of zero locus only sign}, we obtain the following:

\begin{cor}
\label{cor: vertical edge leads to nonsingular}
If $\mathcal{J}_{\lambda} \neq \emptyset$, then $\overline{C_{\lambda}(V)}$ is nonsingular.
\end{cor}

For a vertex $\lambda$ of $\Delta$, we define the cone $\mathcal{C}_{\lambda}:= \{\sum_{i=1}^{n} a_{i}v^{\lambda}_{i} \mid a_{i} \geq 0 \}$.

\begin{prop}
\label{prop: vertical edge leads to vertex}
If $\mathcal{J}_{\lambda} \neq \emptyset$, then $i_{V}^{\ast}(\lambda)$ is a vertex of the convex polytope $i_{V}^{\ast} (\Delta)$.
\end{prop}
\begin{proof}
From Lemma \ref{lemma: basis to basis via linear map}, 
we may assume that the vectors $i_{V}^{\ast}(v^{\lambda}_{1}),\ldots, i_{V}^{\ast}(v^{\lambda}_{n-1})$ are linearly independent. Then, $n \in \mathcal{J}_{\lambda}$, i.e. $i_{V}^{\ast}(v^{\lambda}_{n}) = 0 \in (\mathfrak{t}^{n-1})^{\ast}$.

If there does not exist any nontrivial subspace in the cone $i_{V}^{\ast}(\mathcal{C}_{\lambda})$, then $i_{V}^{\ast}(\lambda)$ is a vertex in the polytope $i_{V}^{\ast}(\Delta)$.
Let $W^{\lambda}_{V}$ be a subspace contained in the cone $i_{V}^{\ast}(\mathcal{C}_{\lambda})$.
For any element $x \in W^{\lambda}_{V} \subset i_{V}^{\ast}(\mathcal{C}_{\lambda})$, there exist $r_{1},\ldots, r_{n} \geq 0$ such that 
\begin{equation*}
    x = \sum_{j=1}^{n}r_{j}i_{V}^{\ast}(v^{\lambda}_{j}) = \sum_{j=1}^{n-1}r_{j}i_{V}^{\ast}(v^{\lambda}_{j}). 
\end{equation*}
Since $W^{\lambda}_{V}$ is a linear space, $(-x) \in W^{\lambda}_{V}$. There exist $s_{1},\ldots, s_{n} \geq 0$ such that 
\begin{equation*}
    -x = \sum_{j=1}^{n}s_{j}i_{V}^{\ast}(v^{\lambda}_{j}) = \sum_{j=1}^{n-1}s_{j}i_{V}^{\ast}(v^{\lambda}_{j}) = \sum_{j=1}^{n-1}(-r_{j})i_{V}^{\ast}(v^{\lambda}_{j}). 
\end{equation*}

Since we assume that $i_{V}^{\ast}(v^{\lambda}_{1}),\ldots,i_{V}^{\ast}(v^{\lambda}_{n-1})$ are linearly independent, we obtain $s_{j} = -r_{j}$ for $j = 1,\ldots, n-1$.
Since $r_{1},\ldots, r_{n-1},s_{1},\ldots,s_{n-1} \geq 0$, we obtain $r_{1}= \cdots = r_{n-1} = s_{1}= \cdots = s_{n-1}= 0$, which means that $x = 0$.
Therefore, we obtain $W^{\lambda}_{V} = \{0\}$, i.e. there is no nontrivial subspace in the cone $i_{V}^{\ast}(\mathcal{C}_{\lambda})$.
The point $i_{V}^{\ast}(\lambda)$ is a vertex of the convex polytope $i_{V}^{\ast}(\Delta)$.
\end{proof}

This proposition tells us that if the vertex $\lambda$ of $\Delta$ does not satisfy the condition (2) in Definition \ref{def: vertex and no direction vectors are perp}, then the vertex $\lambda$ satisfies the condition (1). 

\begin{prop}
\label{prop: vertical edge leads to simple}
If $\mathcal{J}_{\lambda} \neq \emptyset$, then the vectors $i_{V}^{\ast}(v^{\lambda}_{j_{1}}),\ldots, i_{V}^{\ast}(v^{\lambda}_{j_{n-1}})$ for $j_{1},\ldots,j_{n-1} \notin \mathcal{J}_{\lambda}$ are the direction vectors from the vertex $i_{V}^{\ast}(\lambda)$ of the polytope $i_{V}^{\ast}(\Delta)$. 
\end{prop}
\begin{proof}
Let $j \in \mathcal{J}_{\lambda}$, i.e. $i_{V}^{\ast}(v^{\lambda}_{j}) = 0$.
From Lemma \ref{lemma: basis to basis via linear map}, the $(n-1)$ vectors $i_{V}^{\ast}(v^{\lambda}_{1}),\ldots, i_{V}^{\ast}(v^{\lambda}_{j-1}),i_{V}^{\ast}(v^{\lambda}_{j+1}),\ldots, i_{V}^{\ast}(v^{\lambda}_{n})$ are linearly independent.

For the rest of the proof, we show that 
\begin{equation*}
    i_{V}^{\ast}\left(\left\{ \sum_{i=1}^{n}a_{i}v^{\lambda}_{i} \mid a_{i} \geq 0 \right\}\right)
    =
    \left\{ \sum_{i \neq j} b_{i} i_{V}^{\ast}(v^{\lambda}_{i}) \mid b_{i} \geq 0 \right\}.
\end{equation*}
Since it is clear that the right hand side is a subset of the left hand side, we show the other inclusion.
For any $a_{1},\ldots, a_{n} \geq 0$, we obtain
\begin{equation*}
    i_{V}^{\ast}\left(\sum_{i=1}^{n}a_{i}v^{\lambda}_{i} \right)
    =
    \sum_{i=1}^{n}a_{i}i_{V}^{\ast}(v^{\lambda}_{i})
    =
    \sum_{i\neq j}a_{i}i_{V}^{\ast}(v^{\lambda}_{i}).
\end{equation*}
Therefore, $i_{V}^{\ast}(v^{\lambda}_{1}),\ldots, i_{V}^{\ast}(v^{\lambda}_{j-1}),i_{V}^{\ast}(v^{\lambda}_{j+1}),\ldots, i_{V}^{\ast}(v^{\lambda}_{n})$ are the direction vectors from the vertex $i_{V}^{\ast}(\lambda)$.
\end{proof}

\subsection{Case of $i_{V}^{\ast}(\lambda)$ is not a vertex}
\label{subsec: case of not vertex}

Let $\lambda$ be a vertex of the Delzant polytope $\Delta$.
In this section, we show that if the vertex $\lambda$ of $\Delta$ does not satisfy the condition (1) in Definition \ref{def: vertex and no direction vectors are perp}, i.e. there is a nontrivial subspace $W^{\lambda}_{V}$ in the cone $i_{V}^{\ast}(\mathcal{C}_{\lambda})$, then $\overline{C_{\lambda}(V)}$ is nonsingular.
From Proposition \ref{prop: vertical edge leads to vertex}, this vertex $\lambda$ satisfies the condition (2), i.e. $\mathcal{J}_{\lambda} = \emptyset$.

Let $x$ be a nonzero element in the nontrivial subspace $W^{\lambda}_{V}$.
Since $x \in W^{\lambda}_{V}$, we obtain $(-x) \in W^{\lambda}_{V}$.
Since the subspace $W^{\lambda}_{V}$ is in the cone $i_{V}^{\ast}(\mathcal{C}_{\lambda})$, there exist $r_{1},\ldots, r_{n}, s_{1},\ldots, s_{n} \geq 0$ such that 
\begin{equation*}
    x = \sum_{j=1}^{n} r_{j} i_{V}^{\ast}(v^{\lambda}_{j}), \;
    (-x) = \sum_{j=1}^{n} s_{j} i_{V}^{\ast}(v^{\lambda}_{j}).
\end{equation*}
Let $R_{j} = r_{j} + s_{j}$ for $j =1,\ldots, n$. Then, we obtain 
\begin{equation}
    \label{eq: zero sum condition}
    0 =  x + (-x) 
    = \sum_{j=1}^{n}(r_{j}+s_{j})i_{V}^{\ast}(v^{\lambda}_{j})
    = \sum_{j=1}^{n}R_{j}i_{V}^{\ast}(v^{\lambda}_{j}).
\end{equation}

From Lemma \ref{lemma: basis to basis via linear map},
we may assume that the vectors $i_{V}^{\ast}(v^{\lambda}_{1}),\ldots, i_{V}^{\ast}(v^{\lambda}_{n-1})$ are linearly independent.

\begin{lemma}
\label{lemma: not vertex 1}
Assume that the vectors $i_{V}^{\ast}(v^{\lambda}_{1}),\ldots, i_{V}^{\ast}(v^{\lambda}_{n-1})$ are linearly independent.
If $i_{V}^{\ast}(\lambda)$ is not a vertex of the polytope $i_{V}^{\ast}(\Delta)$, then 
$R_{n} \neq 0$ in Equation (\ref{eq: zero sum condition}).
\end{lemma}
\begin{proof}
Assume on the contrary that $R_{n} = 0$.
Then, from Equation (\ref{eq: zero sum condition}) we obtain 
\begin{equation*}
    0= \sum_{j=1}^{n-1}R_{j}i_{V}^{\ast}(v^{\lambda}_{j}).
\end{equation*}
Since the vectors $i_{V}^{\ast}(v^{\lambda}_{1}),\ldots, i_{V}^{\ast}(v^{\lambda}_{n-1})$ are linearly independent, we obtain $R_{1} = \cdots = R_{n-1}=0$.
Since $R_{j} = r_{j} + s_{j}$ and $r_{j},s_{j}\geq 0$, we obtain $r_{j} = s_{j} = 0$ for any $j =1,\ldots, n$, i.e. $x = 0$. This is contradiction to the assumption that $x$ is a nonzero element in the subspace $W^{\lambda}_{V}$.
\end{proof}

\begin{lemma}
\label{lemma: not vertex 2}
Assume that the vectors $i_{V}^{\ast}(v^{\lambda}_{1}),\ldots, i_{V}^{\ast}(v^{\lambda}_{n-1})$ are linearly independent.
If $i_{V}^{\ast}(\lambda)$ is not a vertex of the polytope $i_{V}^{\ast}(\Delta)$, then $\langle u^{\lambda}_{n},q \rangle \neq 0$.
\end{lemma}
\begin{proof}
From Lemma \ref{lemma: not vertex 1}, $R_{n} \neq 0$. Then, Equation (\ref{eq: zero sum condition}) follows 
\begin{equation*}
    i_{V}^{\ast}(v^{\lambda}_{n})
    =
    - \frac{1}{R_{n}}\sum_{j=1}^{n-1}R_{j}i_{V}^{\ast}(v^{\lambda}_{j}).
\end{equation*}
Since from Lemma \ref{lemma: u,q,p,v} we calculate 
\begin{align*}
    \sum_{j=1}^{n-1}\langle u^{\lambda}_{j},q \rangle i_{V}^{\ast}(v^{\lambda}_{j})
    &=
    - \langle u^{\lambda}_{n},q \rangle i_{V}^{\ast}(v^{\lambda}_{n}) \\
    &=
    \langle u^{\lambda}_{n},q \rangle \frac{1}{R_{n}}\sum_{j=1}^{n-1}R_{j}i_{V}^{\ast}(v^{\lambda}_{j}) \\
    &=
    \sum_{j=1}^{n-1}\langle u^{\lambda}_{n},q \rangle \frac{R_{j}}{R_{n}}i_{V}^{\ast}(v^{\lambda}_{j}),
\end{align*}
we obtain 
\begin{equation*}
    \sum_{j=1}^{n-1}\left(\langle u^{\lambda}_{j},q \rangle - \langle u^{\lambda}_{n},q \rangle \frac{R_{j}}{R_{n}} \right)i_{V}^{\ast}(v^{\lambda}_{j}) = 0.
\end{equation*}
Since we assume that the vectors $i_{V}^{\ast}(v^{\lambda}_{1}),\ldots, i_{V}^{\ast}(v^{\lambda}_{n-1})$ are linearly independent, we obtain 
\begin{equation}
    \label{eq: relation formulae}
    \langle u^{\lambda}_{1},q \rangle = \langle u^{\lambda}_{n},q \rangle \frac{R_{1}}{R_{n}}, 
    \ldots,
    \langle u^{\lambda}_{n-1},q \rangle = \langle u^{\lambda}_{n},q \rangle \frac{R_{n-1}}{R_{n}}.
\end{equation}
Since $R_{n} \neq 0$ means that $R_{n} > 0$, we say that 
\begin{equation}
    \label{eq: coefficients of relation formulae}
    \frac{R_{1}}{R_{n}} \geq 0, \ldots, \frac{R_{n-1}}{R_{n}} \geq 0.
\end{equation}

Assume on the contrary that $\langle u^{\lambda}_{n},q \rangle = 0$.
Then, from Equation (\ref{eq: relation formulae}), we obtain $\langle u^{\lambda}_{1},q \rangle = \cdots = \langle u^{\lambda}_{n-1},q \rangle = 0$.
Since $u^{\lambda}_{1},\ldots, u^{\lambda}_{n}$ form a $\mathbb{Z}$-basis of $\mathbb{Z}^{n}$, if $\langle u^{\lambda}_{1},q \rangle = \cdots = \langle u^{\lambda}_{n},q \rangle = 0$, then $q = 0$. This is contradiction to the assumption that $q$ is a basis of the orthogonal subspace $V^{\perp}$ to the $(n-1)$-dimensional subspace $V$ in $\mathfrak{t}^{n} \cong \mathbb{R}^{n}$.
Therefore, $\langle u^{\lambda}_{n},q \rangle \neq 0$.
\end{proof}

\begin{lemma}
\label{lemma: not vertex 3}
If $i_{V}^{\ast}(\lambda)$ is not a vertex of the polytope $i_{V}^{\ast}(\Delta)$, then $\mathcal{I}^{+}_{\lambda} \setminus \mathcal{I}^{0}_{\lambda} = \emptyset$ or $\mathcal{I}^{-}_{\lambda} \setminus \mathcal{I}^{0}_{\lambda} = \emptyset$.
\end{lemma}
\begin{proof}
From Lemma \ref{lemma: basis to basis via linear map}, we may assume that the vectors $i_{V}^{\ast}(v^{\lambda}_{1}),\ldots, i_{V}^{\ast}(v^{\lambda}_{n-1})$ are linearly independent.

By Lemma \ref{lemma: not vertex 2}, we obtain $\langle u^{\lambda}_{n},q \rangle \neq 0$.
Moreover, since Equation (\ref{eq: relation formulae}) and Equation (\ref{eq: coefficients of relation formulae}), if $n \in \mathcal{I}^{+}_{\lambda} \setminus \mathcal{I}^{0}_{\lambda}$, then $1,\ldots,n-1 \in \mathcal{I}^{+}_{\lambda}$, i.e. $\mathcal{I}^{-}_{\lambda}\setminus \mathcal{I}^{0}_{\lambda} = \emptyset$.
Similarly, if $n \in \mathcal{I}^{-}_{\lambda} \setminus \mathcal{I}^{0}_{\lambda}$, then $1,\ldots,n-1 \in \mathcal{I}^{-}_{\lambda}$, i.e. $\mathcal{I}^{+}_{\lambda}\setminus \mathcal{I}^{0}_{\lambda} = \emptyset$.
\end{proof}

From Lemma \ref{lemma: not vertex 3} and Lemma \ref{lemma: rank of zero locus only sign}, we obtain the following:

\begin{cor}
\label{cor: not vertex leads to nonsingular}
If $i_{V}^{\ast}(\lambda)$ is not a vertex of the polytope $i_{V}^{\ast}(\Delta)$, then $\overline{C_{\lambda}(V)}$ is nonsingular.
\end{cor}

\subsection{Case of $\mathcal{J}_{\lambda} = \emptyset$ and $i_{V}^{\ast}(\lambda)$ is a vertex}
\label{subsec: case of vertex}

Let $\lambda$ be a vertex of the Delzant polytope $\Delta$.
In this section, we consider the case when the vertex $\lambda$ of $\Delta$ is a good vertex with respect to the map $i_{V}^{\ast}:(\mathfrak{t}^{n})^{\ast} \to (\mathfrak{t}^{n-1})^{\ast}$ in the sense of Definition \ref{def: vertex and no direction vectors are perp}.
In this case, we suppose the following:
\begin{assumption}
\label{assumption: when good vertex}
The vectors $i_{V}^{\ast}(v^{\lambda}_{1}),\ldots, i_{V}^{\ast}(v^{\lambda}_{n-1}) \in \mathbb{Z}^{n}$ are linearly independent.
Moreover, we have 
\begin{equation*}
    i_{V}^{\ast}\left( \left\{\sum_{i=1}^{n}a_{i}v^{\lambda}_{i}\mid a_{i} \geq 0 \right\}\right) = \left\{\sum_{j=1}^{n-1}b_{j}i_{V}^{\ast}(v^{\lambda}_{j}) \mid b_{j} \geq 0 \right\}.
\end{equation*}
\end{assumption}
This assumption is the same as the first condition in Definition \ref{def: almost delzant}.
Note that if the polytope $i_{V}^{\ast}(\Delta)$ is simple, then Assumption \ref{assumption: when good vertex} holds.

\begin{lemma}
\label{lemma: good vertex leads to splitting}
Under Assumption \ref{assumption: when good vertex},
if the vertex $\lambda$ of $\Delta$ is a good vertex with respect to the map $i_{V}^{\ast}:(\mathfrak{t}^{n})^{\ast} \to (\mathfrak{t}^{n-1})^{\ast}$, then we obtain
\begin{equation*}
    \mathcal{I}^{+}_{\lambda}= \{ 1,\ldots, n-1 \}, \;
    \mathcal{I}^{-}_{\lambda}\setminus \mathcal{I}^{0}_{\lambda} = \{n\}
\end{equation*}
or 
\begin{equation*}
    \mathcal{I}^{+}_{\lambda}\setminus \mathcal{I}^{0}_{\lambda} = \{n\}, \;
    \mathcal{I}^{-}_{\lambda}= \{ 1,\ldots, n-1 \}.
\end{equation*}
\end{lemma}
\begin{proof}
From Assumption \ref{assumption: when good vertex},
there exist $S_{1},\ldots, S_{n-1} \in \mathbb{R}_{\geq 0}$ such that 
\begin{equation}
    \label{eq: cone}
    i_{V}^{\ast}(v^{\lambda}_{n})
    =
    \sum_{j=1}^{n-1}S_{j}i_{V}^{\ast}(v^{\lambda}_{j}).
\end{equation}
From Lemma \ref{lemma: u,q,p,v}, we calculate
\begin{align*}
    0 
    &= 
    \sum_{j=1}^{n}\langle u^{\lambda}_{j},q \rangle i_{V}^{\ast}(v^{\lambda}_{j}) \\
    &=
    \sum_{j=1}^{n-1}\langle u^{\lambda}_{j},q \rangle i_{V}^{\ast}(v^{\lambda}_{j}) + \langle u^{\lambda}_{n},q \rangle i_{V}^{\ast}(v^{\lambda}_{n}) \\
    &=
    \sum_{j=1}^{n-1}\left(\langle u^{\lambda}_{j},q \rangle + S_{j} \langle u^{\lambda}_{n},q \rangle \right) i_{V}^{\ast}(v^{\lambda}_{j}).
\end{align*}
Since $i_{V}^{\ast}(v^{\lambda}_{1}),\ldots, i_{V}^{\ast}(v^{\lambda}_{n-1})$ are linearly independent, we obtain 
\begin{equation}
\label{eq: relation formulae good vertex}
\langle u^{\lambda}_{1},q \rangle = - S_{1}\langle u^{\lambda}_{n},q \rangle, \ldots, \langle u^{\lambda}_{n-1},q \rangle = - S_{n-1}\langle u^{\lambda}_{n},q \rangle.
\end{equation}

Assume on the contrary that $\langle u^{\lambda}_{n},q \rangle = 0$.
If $\langle u^{\lambda}_{n},q \rangle = 0$, then we obtain 
\begin{equation}
    \label{eq: relation formulae good vertex plus assumption}
    \langle u^{\lambda}_{1},q \rangle = \cdots = \langle u^{\lambda}_{n-1},q \rangle = 0
\end{equation}
from Equation (\ref{eq: relation formulae good vertex}).
Since the vectors $u^{\lambda}_{1},\ldots,u^{\lambda}_{n}$ form a $\mathbb{Z}$-basis of $\mathbb{Z}^{n}$, Equation (\ref{eq: relation formulae good vertex plus assumption}) means that $q = 0$, i.e. $V^{\perp} = \{0\}$. This is contradiction to the assumption that $V^{\perp}$ is the orthogonal subspace to the $(n-1)$-dimensional subspace $V$ in $\mathfrak{t}^{n} \cong \mathbb{R}^{n}$.
Therefore, $\langle u^{\lambda}_{n},q \rangle \neq 0$.

Since $\langle u^{\lambda}_{n},q \rangle \neq 0$, we obtain $n \not\in \mathcal{I}^{0}_{\lambda}$, i.e. $n \in \mathcal{I}^{+}_{\lambda} \setminus \mathcal{I}^{0}_{\lambda}$ or $n \in \mathcal{I}^{-}_{\lambda} \setminus \mathcal{I}^{0}_{\lambda}$.
From Equation (\ref{eq: relation formulae good vertex}),
if $n \in \mathcal{I}^{+}_{\lambda} \setminus \mathcal{I}^{0}_{\lambda}$, then $1,\ldots,n-1 \in \mathcal{I}^{-}_{\lambda}$.
Similarly, if $n \in \mathcal{I}^{-}_{\lambda} \setminus \mathcal{I}^{0}_{\lambda}$, then $1,\ldots,n-1 \in \mathcal{I}^{+}_{\lambda}$.
\end{proof}

By the definition of the pullback $i_{V}^{\ast}:(\mathfrak{t}^{n})^{\ast} \to (\mathfrak{t}^{n-1})^{\ast}$, it is clear that the vectors $i_{V}^{\ast}(v^{\lambda}_{1}),\ldots,i_{V}^{\ast}(v^{\lambda}_{n})$ are integral vectors.

\begin{lemma}
\label{lemma: good vertex and almost delzant}
Under Assumption \ref{assumption: when good vertex},
if the vectors $i_{V}^{\ast}(v^{\lambda}_{1}),\ldots, i_{V}^{\ast}(v^{\lambda}_{n-1})$ form a $\mathbb{Z}$-basis of $\mathbb{Z}^{n-1}$, then $\langle u^{\lambda}_{n},q \rangle = \pm 1$.
\end{lemma}
\begin{proof}
Since the vectors $i_{V}^{\ast}(v^{\lambda}_{1}),\ldots, i_{V}^{\ast}(v^{\lambda}_{n-1}) \in \mathbb{Z}^{n-1}$ form a $\mathbb{Z}$-basis of $\mathbb{Z}^{n-1}$ and $i_{V}^{\ast}(v^{\lambda}_{n}) \in \mathbb{Z}^{n-1}$,
there exist $Q_{1},\ldots, Q_{n-1} \in \mathbb{Z}$ such that 
\begin{equation*}
    i_{V}^{\ast}(v^{\lambda}_{n})
    =
    \sum_{j=1}^{n-1}Q_{j}i_{V}^{\ast}(v^{\lambda}_{j}).
\end{equation*}
Moreover, since $i_{V}^{\ast}(v^{\lambda}_{n}) \in i_{V}^{\ast}(\mathcal{C}_{\lambda})$ from Assumption \ref{assumption: when good vertex}, $i_{V}^{\ast}(v^{\lambda}_{n})$ can be expressed as the form in Equation (\ref{eq: cone}).
From these equations, we obtain 
\begin{equation*}
    \sum_{j=1}^{n-1}\left(Q_{j}-S_{j} \right) i_{V}^{\ast}(v^{\lambda}_{j}) = 0.
\end{equation*}
Since the vectors $i_{V}^{\ast}(v^{\lambda}_{1}),\ldots, i_{V}^{\ast}(v^{\lambda}_{n-1})$ are linearly independent, 
\begin{equation}
\label{eq: q equal to s}
    Q_{1}= S_{1}, \ldots, Q_{n-1} = S_{n-1}.
\end{equation}
Since $S_{1},\ldots, S_{n-1} \geq 0$, we have $Q_{1},\ldots, Q_{n-1} \geq 0$.
Thus, we obtain $Q_{1},\ldots, Q_{n-1} \in \mathbb{Z}_{\geq 0}$.

From Equation (\ref{eq: relation formulae good vertex}) and Equation (\ref{eq: q equal to s}), we obtain 
\begin{equation}
\label{eq: relation formulae good vertex and lattice basis}
\langle u^{\lambda}_{1},q \rangle = - Q_{1}\langle u^{\lambda}_{n},q \rangle,\ldots,
\langle u^{\lambda}_{n-1},q \rangle = - Q_{n-1}\langle u^{\lambda}_{n},q \rangle.
\end{equation}
From Equation (\ref{eq: relation formulae good vertex and lattice basis}), we obtain
\begin{equation*}
    {^t\!\left[
        \begin{matrix}
            u^{\lambda}_{1} & \cdots & u^{\lambda}_{n}
        \end{matrix}
    \right]}q 
    =
    \left[
        \begin{matrix}
        \langle u^{\lambda}_{1},q \rangle \\
        \vdots \\
        \langle u^{\lambda}_{n},q \rangle
        \end{matrix}
    \right] 
    =
    \langle u^{\lambda}_{n},q \rangle
    \left[
        \begin{matrix}
        -Q_{1} \\
        \vdots \\
        -Q_{n-1} \\
        1
        \end{matrix}
    \right]
\end{equation*}
and since ${^t[u^{\lambda}_{1}\;\cdots\; u^{\lambda}_{n}]}[v^{\lambda}_{1}\;\cdots\; v^{\lambda}_{n}] = E_{n}$, we obtain 
\begin{equation*}
    q 
    =
    \langle u^{\lambda}_{n},q \rangle
    \left[
        \begin{matrix}
            v^{\lambda}_{1} & \cdots & v^{\lambda}_{n}
        \end{matrix}
    \right]
    \left[
        \begin{matrix}
        -Q_{1} \\
        \vdots \\
        -Q_{n-1} \\
        1
        \end{matrix}
    \right].
\end{equation*}
Since the vector $q \in \mathbb{Z}^{n}$ is primitive,
we obtain $\langle u^{\lambda}_{n},q \rangle = \pm 1$.
\end{proof}

From Lemma \ref{lemma: splitting}, Lemma \ref{lemma: good vertex leads to splitting} and Lemma \ref{lemma: good vertex and almost delzant}, we obtain the following:

\begin{cor}
\label{cor: good vertex and almost delzant lead to nonsingular}
Assume that the vertex $\lambda$ of the Delzant polytope $\Delta$ is a good vertex with respect to the map $i_{V}^{\ast}:(\mathfrak{t}^{n})^{\ast} \to (\mathfrak{t}^{n-1})^{\ast}$.
Under Assumption \ref{assumption: when good vertex}, if there exist distinct $j_{1},\ldots,j_{n-1} \in \{1,\ldots,n\}$ such that the vectors $i_{V}^{\ast}(v^{\lambda}_{j_{1}}),\ldots, i_{V}^{\ast}(v^{\lambda}_{j_{n-1}}) \in \mathbb{Z}^{n-1}$ form a $\mathbb{Z}$-basis of $\mathbb{Z}^{n-1}$, then $\overline{C_{\lambda}(V)}$ is nonsingular.
\end{cor}

\subsection{The Proof of the first Main Theorem}
\label{subsec: proof}

In this section, we give a proof of the first main result.

\begin{thm}
\label{thm: almost delzant to torus equivariant submanifold}
If $V$ is admissible with respect to $\Delta$, then the rank of the Jacobian matrix $Df^{\lambda}$ is equal to one at any point for any vertex $\lambda$ of $\Delta$.
In particular, $\overline{C(V)}$ is a smooth complex hypersurface in the toric manifold $X$ corresponding to $\Delta$.
\end{thm}
\begin{proof}
It is sufficient to consider each vertex of the Delznt polytope $\Delta$.

If a vertex $\lambda$ of $\Delta$ is not a good vertex with respect to the map $i_{V}^{\ast}$, then the vertex $\lambda$ of $\Delta$ does not satisfy the condition (1) or (2) in Definition \ref{def: vertex and no direction vectors are perp}.
If $\lambda$ does not satisfy the condition (2), i.e. $\mathcal{J}_{\lambda} \neq \emptyset$, then Corollary \ref{cor: vertical edge leads to nonsingular} tells us that $\overline{C_{\lambda}(V)}$ is nonsingular.
If $\lambda$ does not satisfy the condition (1), then Corollary \ref{cor: not vertex leads to nonsingular} tells us that $\overline{C_{\lambda}(V)}$ is nonsingular.
Thus, if a vertex $\lambda$ of $\Delta$ is not a good vertex with respect to the map $i_{V}^{\ast}$, then $\overline{C_{\lambda}(V)}$ is nonsingular.

Since we assume that $V$ is admissible with respect to $\Delta$, if a vertex $\lambda$ of $\Delta$ is a good vertex with respect to the map $i_{V}^{\ast}$, then Assumption \ref{assumption: when good vertex} holds for such $\lambda$ and there exist distinct $j_{1},\ldots,j_{n-1} \in \{1,\ldots,n\}$ such that the vectors $i_{V}^{\ast}(v^{\lambda}_{j_{1}}),\ldots, i_{V}^{\ast}(v^{\lambda}_{j_{n-1}}) \in \mathbb{Z}^{n-1}$ form a $\mathbb{Z}$-basis of $\mathbb{Z}^{n-1}$. In Corollary \ref{cor: good vertex and almost delzant lead to nonsingular}, we show that $\overline{C_{\lambda}(V)}$ is nonsingular in this case.

Therefore, if $V$ is admissible with respect to $\Delta$, then $\overline{C_{\lambda}(V)}$ is nonsingular for any vertex $\lambda$ of $\Delta$, i.e. $\overline{C(V)}$ is a smooth complex hypersurface in the toric manifold $X$.
\end{proof}

\section{From Torus-equivariantly Embedded Toric Hypersurfaces to Delzant Polytopes}
\label{sec: torus-equivariant to delzant}

In this section, we show the second main result (Theorem \ref{thm: torus equivariant submanifold to almost delzant}).

Assume that the rank of the Jacobian matrix $Df^{\lambda}$ is equal to one at any point of the zero locus of $f^{\lambda}$ for any vertex $\lambda$ of the Delzant polytope, i.e. $\overline{C(V)}$ is nonsingular.
In this case, as we show in Proposition \ref{prop: toric}, $\overline{C(V)}$ is toric with respect to the Hamiltonian $T^{n-1}$-action defined in \cite{MR4938007} with the moment map $\overline{\mu} = i_{V}^{\ast} \circ \mu\mid_{\overline{C(V)}}$.
Moreover, we have the following:
\begin{thm}[{\cite[Theorem 4.20]{MR4938007}}]
\label{thm: delzant for submanifold}
Assume that $\overline{C(V)}$ is a complex submanifold in the toric manifold $X$.
Then, we obtain 
\begin{equation*}
    \overline{\mu}(\overline{C(V)})
    =
    i_{V}^{\ast}(\Delta),
\end{equation*}
where $\Delta$ is the Delzant polytope of the toric manifold $X$.
\end{thm}

Recall that if $\overline{C(V)}$ is a complex submanifold, then $\overline{C(V)}$ is a symplectic toric manifold with respect to the $T^{n-1}$-action on $\overline{C(V)}$ (Proposition \ref{prop: toric}). 
By the Delzant correspondence, the moment polytope of the submanifold $\overline{C(V)}$ should be Delzant, i.e. we obtain the following:

\begin{cor}
\label{cor: delzant}
Assume that $\overline{C(V)}$ is a complex submanifold in the toric manifold $X$.
Then, the polytope $i_{V}^{\ast}(\Delta)$ is a Delzant polytope.
\end{cor}

In particular, the polytope $i_{V}^{\ast}(\Delta)$ is simple, i.e. Assumption \ref{assumption: when good vertex} holds under the assumption that $\overline{C(V)}$ is a smooth complex hypersurface.

To show the second main result, we have to consider the case when the vertex $\lambda$ of $\Delta$ is good with respect to the map $i_{V}^{\ast}$.

Hereafter, we assume that the vertex $\lambda$ of $\Delta$ is good with respect to the map $i_{V}^{\ast}$.
If $\overline{C(V)}$ is a complex hypersurface, then from Lemma \ref{lemma: good vertex leads to splitting}, we may further assume that 
$\mathcal{I}^{+}_{\lambda} = \{1,\ldots, n-1\}$ and $\mathcal{I}^{-}_{\lambda} \setminus \mathcal{I}^{0}_{\lambda} = \{n\}$ under the Assumption \ref{assumption: when good vertex}.

\begin{lemma}
\label{lemma: simple leads to actual splitting}
Assume that the vertex $\lambda$ of $\Delta$ is good with respect to the map $i_{V}^{\ast}$.
Under the Assumption \ref{assumption: when good vertex}, we obtain $\mathcal{I}^{+}_{\lambda} \setminus \mathcal{I}^{0}_{\lambda} \neq \emptyset$ and $\mathcal{I}^{-}_{\lambda} \setminus \mathcal{I}^{0}_{\lambda} \neq \emptyset$.
\end{lemma}
\begin{proof}
From Lemma \ref{lemma: good vertex leads to splitting}, we may assume that $\mathcal{I}^{+}_{\lambda} = \{1,\ldots, n-1\}$ and $\mathcal{I}^{-}_{\lambda} \setminus \mathcal{I}^{0}_{\lambda} = \{n\}$.

Assume on the contrary that $\mathcal{I}^{+}_{\lambda} \setminus \mathcal{I}^{0}_{\lambda} = \emptyset$, i.e. $\langle u^{\lambda}_{1},q \rangle = \cdots = \langle u^{\lambda}_{n-1},q \rangle = 0$.
From Lemma \ref{lemma: u,q,p,v}, we obtain 
\begin{equation*}
    0 = \sum_{j=1}^{n}\langle u^{\lambda}_{j},q \rangle i_{V}^{\ast} (v^{\lambda}_{j}) 
    = \langle u^{\lambda}_{n},q \rangle i_{V}^{\ast}(v^{\lambda}_{n}).
\end{equation*}
Since $\langle u^{\lambda}_{n},q \rangle \neq 0$, we obtain $i_{V}^{\ast}(v^{\lambda}_{n}) = 0$, which is contradiction to the assumption that the vertex $\lambda$ of $\Delta$ is good with respect to the map $i_{V}^{\ast}$.
\end{proof}

\begin{lemma}
\label{lemma: nonsingular and good vertex lead to two cases}
Assume that the vertex $\lambda$ of $\Delta$ is good with respect to the map $i_{V}^{\ast}$ and $\mathcal{I}^{+}_{\lambda} = \{1,\ldots, n-1\}$ and $\mathcal{I}^{-}_{\lambda} \setminus \mathcal{I}^{0}_{\lambda} = \{n\}$.
If $\overline{C(V)}$ is nonsingular, then we obtain at least one of the following:
\begin{enumerate}
    \item there exists $i_{0} \in \mathcal{I}^{+}_{\lambda}$ such that $\langle u^{\lambda}_{i_{0}},q \rangle = 1$ and $\langle u^{\lambda}_{j},q \rangle = 0$ for any $j \in \mathcal{I}^{+}_{\lambda} \setminus \{i_{0}\}$, 
    \item $\langle u^{\lambda}_{n},q \rangle = -1$.
\end{enumerate}
\end{lemma}
\begin{proof}
From the defining equation for $\overline{C_{\lambda}(V)}$, we obtain that 
\begin{equation*}
    f^{\lambda}(z^{\lambda})
    =
    \prod_{j \in \mathcal{I}^{+}_{\lambda} \setminus \mathcal{I}^{0}_{\lambda}}(z^{\lambda}_{j})^{\langle u^{\lambda}_{j},q \rangle}
    -
    \prod_{j \in \mathcal{I}^{-}_{\lambda} \setminus \mathcal{I}^{0}_{\lambda}}(z^{\lambda}_{j})^{-\langle u^{\lambda}_{j},q \rangle}.
\end{equation*}
Note that $\mathcal{I}^{+}_{\lambda} \setminus \mathcal{I}^{0}_{\lambda} \neq \emptyset$ and $\mathcal{I}^{-}_{\lambda} \setminus \mathcal{I}^{0}_{\lambda} \neq \emptyset$ from Lemma \ref{lemma: simple leads to actual splitting}.
Assume that $\mathcal{I}^{+}_{\lambda} = \{1,\ldots,n-1\}, \; \mathcal{I}^{-}_{\lambda} \setminus \mathcal{I}^{0}_{\lambda} = \{n\}$.
From this equation, we obtain 
\begin{equation*}
    \frac{\partial f^{\lambda}}{\partial z^{\lambda}_{i}} = 
    \begin{cases*}
        \langle u^{\lambda}_{i},q\rangle  \prod_{j \in \mathcal{I}^{+}_{\lambda}\setminus \mathcal{I}^{0}_{\lambda}} (z^{\lambda}_{j})^{\langle u^{\lambda}_{j},q \rangle- \delta_{ij}} & if $i =1,\ldots,n-1$, \\
        \langle u^{\lambda}_{n},q \rangle (z^{\lambda}_{n})^{-\langle u^{\lambda}_{n},q \rangle - 1} & if $i=n$.
    \end{cases*}
\end{equation*}
If $\overline{C(V)}$ is nonsingular, then there exists $i_{0} \in \{ 1,\ldots,n \}$ such that 
\begin{equation*}
    \frac{\partial f^{\lambda}}{\partial z^{\lambda}_{i_{0}}}(o) \neq 0
\end{equation*}
at the origin $o \in \varphi_{\lambda}(U_{\lambda}) \cap \overline{C_{\lambda}(V)}$.
If such $i_{0}$ is an element of $\mathcal{I}^{+}_{\lambda}$, then we obtain $\langle u^{\lambda}_{j},q \rangle - \delta_{i_{0}j} = 0$ for $j =1,\ldots,n-1$, i.e. $\langle u^{\lambda}_{i_{0}},q \rangle = 1$ and $\langle u^{\lambda}_{j},q \rangle = 0$ for any $j \in \mathcal{I}^{+}_{\lambda} \setminus \{i_{0}\}$.
If such $i_{0} \notin \mathcal{I}^{+}_{\lambda}$, i.e. $i_{0} = n$, then $\langle u^{\lambda}_{n},q \rangle = -1$.
\end{proof}

Note that we obtain the similar result to Lemma \ref{lemma: nonsingular and good vertex lead to two cases} when we assume that $\mathcal{I}^{+}_{\lambda}\setminus \mathcal{I}^{0}_{\lambda} = \{ n \}$ and $\mathcal{I}^{-}_{\lambda} = \{ 1, \ldots, n-1\}$.

\subsection{The first case in Lemma \ref{lemma: nonsingular and good vertex lead to two cases}}
We consider the first case in Lemma \ref{lemma: nonsingular and good vertex lead to two cases}.
Hereafter, we assume that $\langle u^{\lambda}_{1},q \rangle = 1$ and $\langle u^{\lambda}_{2},q \rangle = \cdots = \langle u^{\lambda}_{n-1},q \rangle = 0$ for simplicity when we consider the first case in Lemma \ref{lemma: nonsingular and good vertex lead to two cases}.

\begin{lemma}
\label{lemma: simple 1st case 1}
Assume that the vertex $\lambda$ of $\Delta$ is good with respect to the map $i_{V}^{\ast}$.
Moreover, assume that $\langle u^{\lambda}_{1},q \rangle = 1$ and $\langle u^{\lambda}_{2},q \rangle = \cdots = \langle u^{\lambda}_{n-1},q \rangle = 0$.
Then, we obtain 
\begin{equation}
    \label{eq: n is in cone}
    i_{V}^{\ast}(v^{\lambda}_{1})
    =
    - \langle u^{\lambda}_{n},q \rangle i_{V}^{\ast}(v^{\lambda}_{n}).
\end{equation}
\end{lemma}
\begin{proof}
Since we assume that $\langle u^{\lambda}_{1},q \rangle = 1$ and $\langle u^{\lambda}_{2},q \rangle = \cdots = \langle u^{\lambda}_{n-1},q \rangle = 0$, we obtain 
\begin{equation*}
    \sum_{j=1}^{n}\langle u^{\lambda}_{j},q \rangle i_{V}^{\ast}(v^{\lambda}_{j}) 
    = i_{V}^{\ast}(v^{\lambda}_{1}) + \langle u^{\lambda}_{n},q \rangle i_{V}^{\ast}(v^{\lambda}_{n}).
\end{equation*}
From Lemma \ref{lemma: u,q,p,v}, the left-hand side of this equation is equal to zero. 
Therefore, we obtain Equation (\ref{eq: n is in cone}).
\end{proof}

\begin{remark}
\label{remark: re-assumption simple}
If $\lambda$ is good, Assumption \ref{assumption: when good vertex} holds, and $\overline{C_{\lambda}(V)}$ is nonsingular, then we have $i_{V}^{\ast}(v^{\lambda}_{1}) = - \langle u^{\lambda}_{n},q \rangle i_{V}^{\ast}(v^{\lambda}_{n})$ with $-\langle u^{\lambda}_{n},q \rangle \in \mathbb{Z}_{>0}$.
Therefore, we may use the vector $i_{V}^{\ast}(v^{\lambda}_{n})$ instead of $i_{V}^{\ast}(v^{\lambda}_{1})$ in Assumption \ref{assumption: when good vertex}, i.e. we may assume the following:
\begin{itemize}
    \item the vectors $i_{V}^{\ast}(v^{\lambda}_{2}),\ldots, i_{V}^{\ast}(v^{\lambda}_{n})$ are linearly independent,
    \item we have 
    \begin{equation*}
        i_{V}^{\ast}\left( \left\{\sum_{i=1}^{n}a_{i}v^{\lambda}_{i}\mid a_{i} \geq 0 \right\}\right) = \left\{\sum_{j=2}^{n}b_{j}i_{V}^{\ast}(v^{\lambda}_{j}) \mid b_{j} \geq 0 \right\}.
    \end{equation*}
\end{itemize}
\end{remark}

\begin{lemma}
\label{lemma: det of gram and direction vectors 1st case}
Assume the same as in Lemma \ref{lemma: simple 1st case 1}.
Then, we obtain 
\begin{equation*}
    \det [p_{1}\; \cdots \; p_{n-1} \; q]
    =
    \pm
    \langle q, q \rangle
    \det [i_{V}^{\ast}(v^{\lambda}_{2}) \;\cdots \; i_{V}^{\ast}(v^{\lambda}_{n})].
\end{equation*}
\end{lemma}
\begin{proof}
Since from Lemma \ref{lemma: simple 1st case 1}, we obtain 
\begin{align*}
    \left[
        \begin{matrix}
            {^t p_{1}} \\ \vdots \\ {^t p_{n-1}} \\ {^t q} 
        \end{matrix}
    \right]
    \left[
        v^{\lambda}_{1} \; \cdots \; v^{\lambda}_{n}
    \right]
    &=
    \left[
        \begin{matrix}
            \langle p_{1},v^{\lambda}_{1} \rangle & \cdots & \langle p_{1},v^{\lambda}_{n} \rangle \\
            \vdots & \ddots & \vdots \\
            \langle p_{n-1},v^{\lambda}_{1} \rangle & \cdots & \langle p_{n-1},v^{\lambda}_{n} \rangle \\
            \langle q,v^{\lambda}_{1} \rangle & \cdots & \langle q,v^{\lambda}_{n} \rangle
        \end{matrix}
    \right] \\
    &=
    \left[
        \begin{matrix}
            i_{V}^{\ast}(v^{\lambda}_{1}) & \cdots & i_{V}^{\ast}(v^{\lambda}_{n}) \\
            \langle q,v^{\lambda}_{1} \rangle & \cdots & \langle q, v^{\lambda}_{n} \rangle 
        \end{matrix}
    \right] \\
    &=
    \left[
        \begin{matrix}
            -\langle u^{\lambda}_{n},q \rangle i_{V}^{\ast}(v^{\lambda}_{n}) & i_{V}^{\ast}(v^{\lambda}_{2}) & \cdots & i_{V}^{\ast}(v^{\lambda}_{n}) \\
            \langle q,v^{\lambda}_{1} \rangle & \langle q,v^{\lambda}_{2} \rangle & \cdots & \langle q, v^{\lambda}_{n} \rangle 
        \end{matrix}
    \right],
\end{align*}
we obtain by using the cofactor expansion that 
\begin{align*}
    &\det [p_{1}\; \cdots \; p_{n-1} \; q]
    \det [v^{\lambda}_{1} \; \cdots \; v^{\lambda}_{n}] \\
    &=
    \det {^t[p_{1}\; \cdots \; p_{n-1} \; q]}
    \det [v^{\lambda}_{1} \; \cdots \; v^{\lambda}_{n}] \\
    &=
    \det
    \left[
        \begin{matrix}
            -\langle u^{\lambda}_{n},q \rangle i_{V}^{\ast}(v^{\lambda}_{n}) & i_{V}^{\ast}(v^{\lambda}_{2}) & \cdots & i_{V}^{\ast}(v^{\lambda}_{n}) \\
            \langle q,v^{\lambda}_{1} \rangle & \langle q,v^{\lambda}_{2} \rangle & \cdots & \langle q, v^{\lambda}_{n} \rangle 
        \end{matrix}
    \right] \\
    &=
    \det
    \left[
        \begin{matrix}
            0 &i_{V}^{\ast}(v^{\lambda}_{2}) & \cdots & i_{V}^{\ast}(v^{\lambda}_{n}) \\
            \langle q,v^{\lambda}_{1} \rangle + \langle q,v^{\lambda}_{n} \rangle \langle u^{\lambda}_{n},q \rangle & \langle q, v^{\lambda}_{2} \rangle & \cdots & \langle q, v^{\lambda}_{n} \rangle 
        \end{matrix}
    \right] \\
    &=
    \pm \left(
        \langle q,v^{\lambda}_{1} \rangle + \langle q,v^{\lambda}_{n} \rangle \langle u^{\lambda}_{n},q \rangle
    \right)
    \det [i_{V}^{\ast}(v^{\lambda}_{2}) \;\cdots \; i_{V}^{\ast}(v^{\lambda}_{n})].
\end{align*}
Since we assume $\langle u^{\lambda}_{1},q \rangle = 1$ and $\langle u^{\lambda}_{2},q \rangle = \cdots = \langle u^{\lambda}_{n-1},q \rangle = 0$, we obtain 
\begin{equation*}
    \langle q,v^{\lambda}_{1} \rangle + \langle q,v^{\lambda}_{n} \rangle \langle u^{\lambda}_{n},q \rangle
    =
    \sum_{j=1}^{n}\langle q,v^{\lambda}_{j} \rangle \langle u^{\lambda}_{j},q \rangle
    =
    \langle q,q \rangle.
\end{equation*}
Thus, we obtain 
\begin{align*}
    &\pm \left(
        \langle q,v^{\lambda}_{1} \rangle + \langle q,v^{\lambda}_{n} \rangle \langle u^{\lambda}_{n},q \rangle
    \right)
    \det [i_{V}^{\ast}(v^{\lambda}_{2}) \;\cdots \; i_{V}^{\ast}(v^{\lambda}_{n})]
    \\
    &= 
    \pm \langle q,q \rangle \det [i_{V}^{\ast}(v^{\lambda}_{2}) \;\cdots \; i_{V}^{\ast}(v^{\lambda}_{n})]. 
\end{align*}
Since $v^{\lambda}_{1},\ldots, v^{\lambda}_{n}$ form a basis of $\mathbb{Z}^{n}$, $\det [v^{\lambda}_{1} \; \cdots \; v^{\lambda}_{n}] = \pm 1$, and 
\begin{equation*}
    \det [p_{1}\; \cdots \; p_{n-1} \; q]
    =
    \pm \langle q,q \rangle \det [i_{V}^{\ast}(v^{\lambda}_{2}) \;\cdots \; i_{V}^{\ast}(v^{\lambda}_{n})].
\end{equation*}
\end{proof}

\subsection{The second case in Lemma \ref{lemma: nonsingular and good vertex lead to two cases}}

Next, we consider the second case in Lemma \ref{lemma: nonsingular and good vertex lead to two cases}, i.e.
we assume that $\langle u^{\lambda}_{n},q \rangle = -1$ and $\mathcal{I}^{+}_{\lambda} = \{1,\ldots,n-1 \}$.

\begin{lemma}
\label{lemma: simple 2nd case}
Assume that the vertex $\lambda$ of $\Delta$ is good with respect to the map $i_{V}^{\ast}$ and that $\langle u^{\lambda}_{n},q \rangle = -1$ and $\mathcal{I}^{+}_{\lambda} = \{1,\ldots,n-1 \}$.
Then, we obtain 
\begin{equation}
    \label{eq: n is in cone 2nd case}
    i_{V}^{\ast}(v^{\lambda}_{n})
    = \sum_{j=1}^{n-1} \langle u^{\lambda}_{j},q \rangle i_{V}^{\ast}(v^{\lambda}_{j}).
\end{equation}
\end{lemma}
\begin{proof}
From Lemma \ref{lemma: u,q,p,v}, we obtain 
\begin{align*}
    0= \sum_{j=1}^{n} \langle u^{\lambda}_{j},q \rangle i_{V}^{\ast}(v^{\lambda}_{j})
    =
    \sum_{j=1}^{n-1} \langle u^{\lambda}_{j},q \rangle i_{V}^{\ast}(v^{\lambda}_{j}) - i_{V}^{\ast}(v^{\lambda}_{n}).
\end{align*}
Thus, we obtain Equation (\ref{eq: n is in cone 2nd case}).
\end{proof}

\begin{lemma}
\label{lemma: det of gram and direction vectors}
Assume that $\langle u^{\lambda}_{n},q \rangle = -1$ and $\mathcal{I}^{+}_{\lambda} = \{1,\ldots,n-1 \}$.
Then, we obtain 
\begin{equation*}
    \det [p_{1}\; \cdots \; p_{n-1} \; q]
    =
    \pm
    \langle q, q \rangle
    \det [i_{V}^{\ast}(v^{\lambda}_{1}) \;\cdots \; i_{V}^{\ast}(v^{\lambda}_{n-1})].
\end{equation*}
\end{lemma}
\begin{proof}
Since we see 
\begin{align*}
    \left[
        \begin{matrix}
            {^t p_{1}} \\ \vdots \\ {^t p_{n-1}} \\ {^t q} 
        \end{matrix}
    \right]
    \left[
        v^{\lambda}_{1} \; \cdots \; v^{\lambda}_{n}
    \right]
    =
    \left[
        \begin{matrix}
            i_{V}^{\ast}(v^{\lambda}_{1}) & \cdots & i_{V}^{\ast}(v^{\lambda}_{n}) \\
            \langle q,v^{\lambda}_{1} \rangle & \cdots & \langle q, v^{\lambda}_{n} \rangle 
        \end{matrix}
    \right],
\end{align*}
we obtain by using the cofactor expansion that
\begin{align*}
    &\det [p_{1}\; \cdots \; p_{n-1} \; q]
    \det [v^{\lambda}_{1} \; \cdots \; v^{\lambda}_{n}] \\
    &=
    \det 
    \left[
        \begin{matrix}
            i_{V}^{\ast}(v^{\lambda}_{1}) & \cdots & i_{V}^{\ast}(v^{\lambda}_{n}) \\
            \langle q,v^{\lambda}_{1} \rangle & \cdots & \langle q, v^{\lambda}_{n} \rangle 
        \end{matrix}
    \right] \\
    &=
    \det 
    \left[
        \begin{matrix}
            i_{V}^{\ast}(v^{\lambda}_{1}) & \cdots & i_{V}^{\ast}(v^{\lambda}_{n-1}) & \sum_{j=1}^{n-1}\langle u^{\lambda}_{j},q \rangle i_{V}^{\ast}(v^{\lambda}_{j}) \\
            \langle q,v^{\lambda}_{1} \rangle & \cdots & \langle q,v^{\lambda}_{n-1} \rangle & \langle q, v^{\lambda}_{n} \rangle 
        \end{matrix}
    \right] \\
    &=
    \det 
    \left[
        \begin{matrix}
            i_{V}^{\ast}(v^{\lambda}_{1}) & \cdots & i_{V}^{\ast}(v^{\lambda}_{n-1}) & 0 \\
            \langle q,v^{\lambda}_{1} \rangle & \cdots & \langle q, v^{\lambda}_{n-1} \rangle & \langle q,v^{\lambda}_{n} \rangle - \sum_{j=1}^{n-1}\langle q,v^{\lambda}_{j} \rangle \langle u^{\lambda}_{j},q \rangle
        \end{matrix}
    \right] \\
    &=
    \left(
        \langle q,v^{\lambda}_{n} \rangle - \sum_{j=1}^{n-1}\langle q,v^{\lambda}_{j} \rangle \langle u^{\lambda}_{j},q \rangle
    \right)
    \det [i_{V}^{\ast}(v^{\lambda}_{1}) \;\cdots \; i_{V}^{\ast}(v^{\lambda}_{n-1})]. 
\end{align*}
Since we assume $\langle u^{\lambda}_{n},q \rangle = -1$, we obtain 
\begin{equation*}
    \langle q,v^{\lambda}_{n} \rangle - \sum_{j=1}^{n-1}\langle q,v^{\lambda}_{j} \rangle \langle u^{\lambda}_{j},q \rangle
    =
    - \sum_{j=1}^{n}\langle q,v^{\lambda}_{j} \rangle \langle u^{\lambda}_{j},q \rangle
    =
    -\langle q,q \rangle.
\end{equation*}
Thus, we obtain 
\begin{align*}  
    &\left(
        \langle q,v^{\lambda}_{n} \rangle - \sum_{j=1}^{n-1}\langle q,v^{\lambda}_{j} \rangle \langle u^{\lambda}_{j},q \rangle
    \right)
    \det [i_{V}^{\ast}(v^{\lambda}_{1}) \;\cdots \; i_{V}^{\ast}(v^{\lambda}_{n-1})] \\
    &=
    -\langle q,q \rangle \det [i_{V}^{\ast}(v^{\lambda}_{1}) \;\cdots \; i_{V}^{\ast}(v^{\lambda}_{n-1})]. 
\end{align*}
Since $v^{\lambda}_{1},\ldots, v^{\lambda}_{n}$ form a basis of $\mathbb{Z}^{n}$, $\det [v^{\lambda}_{1} \; \cdots \; v^{\lambda}_{n}] = \pm 1$, and 
\begin{equation*}
    \det [p_{1}\; \cdots \; p_{n-1} \; q]
    =
    \pm
    \langle q, q \rangle
    \det [i_{V}^{\ast}(v^{\lambda}_{1}) \;\cdots \; i_{V}^{\ast}(v^{\lambda}_{n-1})].
\end{equation*}
\end{proof}

\subsection{The Proof of the second Main Theorem}

By calculating the determinant of the Gram matrix of $[p_{1}\; \cdots \;p_{n-1} \;q]$, we obtain the following:
\begin{lemma}
\label{lemma: det of gram matrix}
For $V = \mathbb{R}p_{1}+\cdots+\mathbb{R}p_{n-1}$ and $V^{\perp}=\mathbb{R}q$, we obtain
\begin{equation*}
    \det [p_{1}\; \cdots \; p_{n-1} \; q]
    =
    \pm \left(
        \langle q, q \rangle  
        \det \left[
            \begin{matrix}
                \langle p_{1},p_{1} \rangle &\cdots & \langle p_{1},p_{n-1} \rangle \\
                \vdots & \ddots &\vdots \\
                \langle p_{n-1},p_{1} \rangle & \cdots & \langle p_{n-1},p_{n-1} \rangle 
            \end{matrix} 
        \right]
    \right)^{\frac{1}{2}}.
\end{equation*}
\end{lemma}
\begin{proof}
Since the Gram matrix of $[p_{1}\; \cdots \;p_{n-1} \;q]$ is 
\begin{align*}
    \left[
        \begin{matrix}
            {^t p_{1}} \\
            \vdots \\
            {^t p_{n-1}} \\
            {^t q}
        \end{matrix}    
    \right]
    [p_{1}\; \cdots \;p_{n-1} \;q]
    &=
    \left[
        \begin{matrix}
            \langle p_{1},p_{1} \rangle & \cdots & \langle p_{1},p_{n-1} \rangle & \langle p_{1},q \rangle \\
            \vdots & \ddots & \vdots &\vdots \\ 
            \langle p_{n-1},p_{1} \rangle & \cdots & \langle p_{n-1},p_{n-1} \rangle & \langle p_{1},q \rangle \\
            \langle q,p_{1}\rangle & \cdots & \langle q,p_{n-1} \rangle & \langle q,q \rangle 
        \end{matrix}
    \right] \\
    &=
    \left[
        \begin{matrix}
            \langle p_{1},p_{1} \rangle & \cdots & \langle p_{1},p_{n-1} \rangle & 0 \\
            \vdots & \ddots & \vdots &\vdots \\ 
            \langle p_{n-1},p_{1} \rangle & \cdots & \langle p_{n-1},p_{n-1} \rangle & 0 \\
            0 & \cdots & 0 & \langle q,q \rangle 
        \end{matrix}
    \right],
\end{align*}
we obtain 
\begin{align*}
    (\det [p_{1}\; \cdots \; p_{n-1} \; q])^{2}
    &=
    \det \left(
        \left[
        \begin{matrix}
            {^t p_{1}} \\
            \vdots \\
            {^t p_{n-1}} \\
            {^t q}
        \end{matrix}    
    \right]
    [p_{1}\; \cdots \;p_{n-1} \;q]
    \right) \\
    &=
    \det 
    \left[
        \begin{matrix}
            \langle p_{1},p_{1} \rangle & \cdots & \langle p_{1},p_{n-1} \rangle & 0 \\
            \vdots & \ddots & \vdots &\vdots \\ 
            \langle p_{n-1},p_{1} \rangle & \cdots & \langle p_{n-1},p_{n-1} \rangle & 0 \\
            0 & \cdots & 0 & \langle q,q \rangle 
        \end{matrix}
    \right] \\
    &=
    \langle q, q \rangle  
    \det \left[
        \begin{matrix}
            \langle p_{1},p_{1} \rangle &\cdots & \langle p_{1},p_{n-1} \rangle \\
            \vdots & \ddots &\vdots \\
            \langle p_{n-1},p_{1} \rangle & \cdots & \langle p_{n-1},p_{n-1} \rangle 
        \end{matrix} 
    \right].
\end{align*}
Since Gram matrices are positive definite, i.e. their determinant are nonnegative, we obtain the desired result.
\end{proof}

In our case, we obtain more about the determinant of Gram matrices.
Though the following proposition is found in books \cite[Proposition 1.9.8]{MR1957723} and \cite[Proposition 1.2]{MR2977354}, we give a proof in order to make this paper self-contained.
\begin{prop}
\label{prop: det of gram matrices}
Assume that $p_{1},\ldots,p_{n-1} \in \mathbb{Z}^{n}$ form a $\mathbb{Z}$-basis of $V \cap \mathbb{Z}^{n} \cong \mathbb{Z}^{n-1}$.
Let $q \in \mathbb{Z}^{n}$ be a primitive basis of the orthogonal subspace to $V$ in $\mathbb{R}^{n}$.
For $V = \mathbb{R}p_{1}+\cdots+\mathbb{R}p_{n-1}$ and $V^{\perp}=\mathbb{R}q$, we obtain
\begin{equation*}
    \det \left[
            \begin{matrix}
                \langle p_{1},p_{1} \rangle &\cdots & \langle p_{1},p_{n-1} \rangle \\
                \vdots & \ddots &\vdots \\
                \langle p_{n-1},p_{1} \rangle & \cdots & \langle p_{n-1},p_{n-1} \rangle 
            \end{matrix} 
        \right]
        =
        \langle q,q \rangle.
\end{equation*}
\end{prop}
\begin{proof}
By the assumption, there exists a vector $p_{n} \in \mathbb{Z}^{n}$ such that $p_{1},\ldots, p_{n-1},p_{n}$ form a $\mathbb{Z}$-basis of $\mathbb{Z}^{n}$. 
Since we have
\begin{align*}
    \left[
        \begin{matrix}
            {^t p_{1}} \\ \vdots \\ {^t p_{n-1}} \\ {^t p_{n}} 
        \end{matrix}
    \right]
    \left[
        p_{1} \; \cdots \; p_{n-1} \; q
    \right]
    &=
    \left[
        \begin{matrix}
            \langle p_{1},p_{1}\rangle & \cdots &\langle p_{1},p_{n-1} \rangle & \langle p_{1},q \rangle \\
            \vdots & \ddots & \vdots & \vdots \\ 
            \langle p_{n-1},p_{1} \rangle & \cdots & \langle p_{n-1},p_{n-1} \rangle & \langle p_{n-1},q \rangle \\
            \langle p_{n},p_{1}\rangle & \cdots & \langle p_{n},p_{n-1} \rangle & \langle p_{n},q \rangle 
        \end{matrix}
    \right]\\
    &=
    \left[
        \begin{matrix}
            \langle p_{1},p_{1}\rangle & \cdots &\langle p_{1},p_{n-1} \rangle & 0 \\
            \vdots & \ddots & \vdots & \vdots \\ 
            \langle p_{n-1},p_{1} \rangle & \cdots & \langle p_{n-1},p_{n-1} \rangle & 0 \\
            \langle p_{n},p_{1}\rangle & \cdots & \langle p_{n},p_{n-1} \rangle & \langle p_{n},q \rangle 
        \end{matrix}
    \right],
\end{align*}
we obtain by the cofactor expansion that
\begin{align*}
    &\det [p_{1}\; \cdots \; p_{n-1} \; p_{n}]
    \det [p_{1} \; \cdots \; p_{n-1} \; q] \\
    &=
    \det 
    \left[
        \begin{matrix}
            \langle p_{1},p_{1}\rangle & \cdots &\langle p_{1},p_{n-1} \rangle & 0 \\
            \vdots & \ddots & \vdots & \vdots \\ 
            \langle p_{n-1},p_{1} \rangle & \cdots & \langle p_{n-1},p_{n-1} \rangle & 0 \\
            \langle p_{n},p_{1}\rangle & \cdots & \langle p_{n},p_{n-1} \rangle & \langle p_{n},q \rangle 
        \end{matrix}
    \right] \\
    &= 
    \langle p_{n},q \rangle
    \det 
    \left[
        \begin{matrix}
            \langle p_{1},p_{1}\rangle & \cdots &\langle p_{1},p_{n-1} \rangle  \\
            \vdots & \ddots & \vdots \\ 
            \langle p_{n-1},p_{1} \rangle & \cdots & \langle p_{n-1},p_{n-1} \rangle 
        \end{matrix}
    \right]. 
\end{align*}
Since $p_{1},\ldots,p_{n}$ form a $\mathbb{Z}$-basis of $\mathbb{Z}^{n}$, we obtain $\det [p_{1}\; \cdots \; p_{n-1} \; p_{n}] = \pm 1$.
From Lemma \ref{lemma: det of gram matrix}, we obtain 
\begin{equation}
\label{eq: gram matrix identity}
    \langle p_{n},q \rangle^{2}
    \det 
    \left[
        \begin{matrix}
            \langle p_{1},p_{1}\rangle & \cdots &\langle p_{1},p_{n-1} \rangle  \\
            \vdots & \ddots & \vdots \\ 
            \langle p_{n-1},p_{1} \rangle & \cdots & \langle p_{n-1},p_{n-1} \rangle 
        \end{matrix}
    \right]
    = \langle q,q \rangle.
\end{equation}
Since we have 
\begin{equation*}
    \left[
        \begin{matrix}
            \langle p_{1},q \rangle \\
            \vdots \\
            \langle p_{n-1},q \rangle \\
            \langle p_{n},q \rangle 
        \end{matrix}
    \right]
    =
    \left[
        \begin{matrix}
            {^t p_{1}} \\
            \vdots \\
            {^t p_{n-1}} \\
            {^t p_{n}}
        \end{matrix}
    \right] q,
\end{equation*}
and $p_{1},\ldots,p_{n}$ form a $\mathbb{Z}$-basis of $\mathbb{Z}^{n}$, we obtain 
\begin{align*}
    q 
    &=
    \left[
        \begin{matrix}
            {^t p_{1}} \\
            \vdots \\
            {^t p_{n-1}} \\
            {^t p_{n}}
        \end{matrix}
    \right]^{-1}
    \left[
        \begin{matrix}
            \langle p_{1},q \rangle \\
            \vdots \\
            \langle p_{n-1},q \rangle \\
            \langle p_{n},q \rangle 
        \end{matrix}
    \right] \\
    &=
    \left[
        \begin{matrix}
            {^t p_{1}} \\
            \vdots \\
            {^t p_{n-1}} \\
            {^t p_{n}}
        \end{matrix}
    \right]^{-1}
    \left[
        \begin{matrix}
            0 \\
            \vdots \\
            0 \\
            \langle p_{n},q \rangle 
        \end{matrix}
    \right] \\
    &=
    \langle p_{n},q \rangle 
    \left[
        \begin{matrix}
            {^t p_{1}} \\
            \vdots \\
            {^t p_{n-1}} \\
            {^t p_{n}}
        \end{matrix}
    \right]^{-1}
    \left[
        \begin{matrix}
            0 \\
            \vdots \\
            0 \\
            1
        \end{matrix}
    \right].
\end{align*}
Since we assume that the vector $q$ is primitive, we obtain $\langle p_{n},q \rangle = \pm1$.
Therefore, from Equation (\ref{eq: gram matrix identity}) we obtain the desired equation.
\end{proof}

From Lemma \ref{lemma: det of gram matrix} and Proposition \ref{prop: det of gram matrices}, we obtain the following:

\begin{cor}
\label{cor: det of gram matrices}
Assume that $p_{1},\ldots,p_{n-1} \in \mathbb{Z}^{n}$ form a $\mathbb{Z}$-basis of $V \cap \mathbb{Z}^{n} \cong \mathbb{Z}^{n-1}$.
Assume that the vector $q \in \mathbb{Z}^{n}$ is a primitive basis of the orthogonal subspace to $V$ in $\mathbb{R}^{n}$.
For $V = \mathbb{R}p_{1}+\cdots+\mathbb{R}p_{n-1}$ and $V^{\perp}=\mathbb{R}q$, we obtain 
\begin{equation*}
    \det [p_{1}\; \cdots \; p_{n-1} \; q]
    =
    \pm \langle q, q \rangle .
\end{equation*}
\end{cor}

We use this fact to show that the second main result:
\begin{thm}
\label{thm: torus equivariant submanifold to almost delzant}
If the rank of the Jacobian matrix $Df^{\lambda}$ is equal to one at any point of the zero locus of $f^{\lambda}$ for any vertex $\lambda$ of the Delzant polytope $\Delta$, i.e. $\overline{C(V)}$ is nonsingular, then $V$ is admissible with respect to $\Delta$.
\end{thm}
\begin{proof}
It is sufficient to consider each vertex $\lambda$ of the Delznt polytope $\Delta$.

From Lemma \ref{lemma: good vertex leads to splitting} and Lemma \ref{lemma: nonsingular and good vertex lead to two cases}, if the vertex $\lambda$ of $\Delta$ is good with respect to the map $i_{V}^{\ast}$, then we obtain the two cases:
\begin{enumerate}
    \item there exists $i_{0} \in \mathcal{I}^{+}_{\lambda}$ such that $\langle u^{\lambda}_{i_{0}},q \rangle = 1$ and $\langle u^{\lambda}_{j},q \rangle = 0$ for any $j \in \mathcal{I}^{+}_{\lambda} \setminus \{i_{0}\}$, 
    \item $\langle u^{\lambda}_{n},q \rangle = -1$.
\end{enumerate}

We first consider the first case.
In this case, we can assume further that $\langle u^{\lambda}_{1},q \rangle = 1$ and $\langle u^{\lambda}_{2},q \rangle = \cdots = \langle u^{\lambda}_{n-1},q \rangle = 0$.
Under this assumption, Remark \ref{remark: re-assumption simple} shows that the vectors $i_{V}^{\ast}(v^{\lambda}_{2}), \ldots, i_{V}^{\ast}(v^{\lambda}_{n})$ are linearly independent and 
\begin{equation*}
        i_{V}^{\ast}\left( \left\{\sum_{i=1}^{n}a_{i}v^{\lambda}_{i}\mid a_{i} \geq 0 \right\}\right) = \left\{\sum_{j=2}^{n}b_{j}i_{V}^{\ast}(v^{\lambda}_{j}) \mid b_{j} \geq 0 \right\}.
\end{equation*} 
From Lemma \ref{lemma: det of gram and direction vectors 1st case} and Corollary \ref{cor: det of gram matrices}, we obtain 
\begin{equation*}
    \det 
    [i_{V}^{\ast}(v^{\lambda}_{2}) \; \cdots i_{V}^{\ast}(v^{\lambda}_{n})]
    = \pm1,
\end{equation*}
which shows that the vectors $i_{V}^{\ast}(v^{\lambda}_{2}), \ldots, i_{V}^{\ast}(v^{\lambda}_{n})$ form a $\mathbb{Z}$-basis of $\mathbb{Z}^{n-1}$.

We next consider the second case. 
In this case, from Lemma \ref{lemma: det of gram and direction vectors} and Corollary \ref{cor: det of gram matrices}, we obtain 
\begin{equation*}
    \det 
    [i_{V}^{\ast}(v^{\lambda}_{1}) \; \cdots i_{V}^{\ast}(v^{\lambda}_{n-1})]
    = \pm1,
\end{equation*}
which shows that the vectors $i_{V}^{\ast}(v^{\lambda}_{1}), \ldots, i_{V}^{\ast}(v^{\lambda}_{n-1})$ form a $\mathbb{Z}$-basis of $\mathbb{Z}^{n-1}$.

Therefore, $V$ is admissible with respect to $\Delta$.
\end{proof}

\section{Examples} \label{sec: example}
We demonstrate some examples to illustrate admissible subspaces with respect to given Delzant polytopes (Definition \ref{def: almost delzant}).
Let $\Delta_{n} \subset (\mathfrak{t}^{n})^{\ast}$ denote the standard $n$-simplex,
i.e.
$\Delta_{n}$ is the convex hull of the points
\begin{equation*}
    \sigma_{0} := (0,\ldots,0), \;
    \sigma_{1} := (1,0,\ldots,0), \ldots,
    \sigma_{n} := (0,\ldots,0,1).
\end{equation*}
Let $e_{1},\ldots,e_{n}$ denote the standard basis of $(\mathfrak{t}^{n})^{\ast}$.

For the rest of this section, we use $\Delta_{2}$ and $\Delta_{4}$.
When $n=2$, the direction vectors $v^{i}_{1},\ldots, v^{i}_{n}$ of the vertex $\sigma_{i}$ for $i = 0,1,2$ is given by 
\begin{equation*}
    v^{0}_{1} = e_{1}, \; v^{0}_{2} = e_{2}, \;
    v^{1}_{1} = -e_{1} + e_{2}, \; v^{1}_{2} = -e_{1}, \;
    v^{2}_{1} = - e_{2}, \; v^{2}_{2} = e_{1} - e_{2}.
\end{equation*}
When $n=4$, the direction vectors $v^{i}_{1},\ldots, v^{i}_{n}$ of the vertex $\sigma_{i}$ for $i = 0,\ldots,4$ is given by 
\begin{align*}
    &v^{0}_{1} = e_{1}, \; v^{0}_{2} = e_{2}, \; v^{0}_{3} = e_{3}, \; v^{0}_{4} = e_{4}, \\
    &v^{1}_{1} = -e_{1}, \; v^{1}_{2} = e_{2} - e_{1}, \; v^{1}_{3} = e_{3} - e_{1}, \; v^{1}_{4} = e_{4} - e_{1}, \\
    &v^{2}_{1} = - e_{2}, \; v^{2}_{2} = e_{1} - e_{2},\; v^{2}_{3} = e_{3} - e_{2}, \; v^{2}_{4} = e_{4} - e_{2}, \\
    &v^{3}_{1} = -e_{3}, \; v^{3}_{2} = e_{1} - e_{3}, \; v^{3}_{3} = e_{2} - e_{3}, \; v^{3}_{4} = e_{4} - e_{3}, \\ 
    &v^{4}_{1} = -e_{4}, \; v^{4}_{2} = e_{1} - e_{4}, \; v^{4}_{3} = e_{2} - e_{4}, \; v^{4}_{4} = e_{3} - e_{4}.
\end{align*}

\begin{example}
\label{example: slope 1 in R^2}
Let $V \subset \mathfrak{t}^{2}$ be a one-dimensional subspace generated by the vector $e_{1}+e_{2}$. 
Then $V$ is admissible with respect to the standard two-simplex $\Delta_{2}$.
Indeed, 
since the map $i_{V}^{\ast}:(\mathfrak{t}^{2})^{\ast} \to (\mathfrak{t}^{1})^{\ast}$ is given by $i_{V}^{\ast}(\xi_{1},\xi_{2}) = \xi_{1} + \xi_{2}$,
we see that 
\begin{align*}
    &i_{V}^{\ast}(v^{0}_{1}) = 1, \; i_{V}^{\ast}(v^{0}_{2}) = 1, \\
    &i_{V}^{\ast}(v^{1}_{1}) = 0, \; i_{V}^{\ast}(v^{1}_{2}) = -1, \\
    &i_{V}^{\ast}(v^{2}_{1}) = -1, \; i_{V}^{\ast}(v^{2}_{2}) = 0.
\end{align*}
Hence, the vertex $\sigma_{0} \in \Delta_{2}$ is good with respect to the map $i_{V}^{\ast}$ while the vertices $\sigma_{1},\sigma_{2} \in \Delta_{2}$ are not good with respect to the map $i_{V}^{\ast}$ because the vertices $\sigma_{1},\sigma_{2}$ do not satisfy the condition (2) in Definition \ref{def: vertex and no direction vectors are perp}.
Moreover, since the polytope $i_{V}^{\ast}(\Delta_{2})$ is an interval, $V$ is admissible with respect to the standard two-simplex $\Delta_{2}$.
\end{example}

\begin{example}
\label{example: slope 2 in R^2}
Let $V \subset \mathfrak{t}^{2}$ be a one-dimensional subspace  generated by the vector $e_{1}+2 e_{2}$. 
Then $V$ is admissible with respect to the standard two-simplex $\Delta_{2}$.
Indeed, 
since the map $i_{V}^{\ast}:(\mathfrak{t}^{2})^{\ast} \to (\mathfrak{t}^{1})^{\ast}$ is given by $i_{V}^{\ast}(\xi_{1},\xi_{2}) = \xi_{1} + 2\xi_{2}$,
we see that 
\begin{align*}
    &i_{V}^{\ast}(v^{0}_{1}) = 1, \; i_{V}^{\ast}(v^{0}_{2}) = 2, \\
    &i_{V}^{\ast}(v^{1}_{1}) = 1, \; i_{V}^{\ast}(v^{1}_{2}) = -1, \\
    &i_{V}^{\ast}(v^{2}_{1}) = -2, \; i_{V}^{\ast}(v^{2}_{2}) = -1.
\end{align*}
Hence, the vertices $\sigma_{0},\sigma_{2} \in \Delta_{2}$ are good with respect to the map $i_{V}^{\ast}$ while the vertex $\sigma_{1} \in \Delta_{2}$ is not good with respect to the map $i_{V}^{\ast}$ because the vertex $\sigma_{1}$ do not satisfy the condition (1) in Definition \ref{def: vertex and no direction vectors are perp}.
Moreover, since the polytope $i_{V}^{\ast}(\Delta_{2})$ is an interval, $V$ is admissible with respect to the standard two-simplex $\Delta_{2}$.
\end{example}

\begin{example}
\label{example: slope 3 in R^2}
Let $V \subset \mathfrak{t}^{2}$ be a one-dimensional subspace generated by the vector $e_{1}+3 e_{2}$. 
Then $V$ is not admissible with respect to the standard two-simplex $\Delta_{2}$ because $V$ fails to meet the second condition in Definition \ref{def: almost delzant}. 
Indeed, 
since the map $i_{V}^{\ast}:(\mathfrak{t}^{2})^{\ast} \to (\mathfrak{t}^{1})^{\ast}$ is given by $i_{V}^{\ast}(\xi_{1},\xi_{2}) = \xi_{1} + 3\xi_{2}$,
we see that 
\begin{align*}
    &i_{V}^{\ast}(v^{0}_{1}) = 1, \; i_{V}^{\ast}(v^{0}_{2}) = 3, \\
    &i_{V}^{\ast}(v^{1}_{1}) = 2, \; i_{V}^{\ast}(v^{1}_{2}) = -1, \\
    &i_{V}^{\ast}(v^{2}_{1}) = -3, \; i_{V}^{\ast}(v^{2}_{2}) = -2.
\end{align*}
Hence, the vertices $sigma_{0},\sigma_{2} \in \Delta_{2}$ are good with respect to the map $i_{V}^{\ast}$ while the vertex $\sigma_{1} \in \Delta_{2}$ is not good with respect to the map $i_{V}^{\ast}$ because the vertex $\sigma_{1}$ do not satisfy the condition (1) in Definition \ref{def: vertex and no direction vectors are perp}.
Moreover, 
while the polytope $i_{V}^{\ast}(\Delta_{2})$ is an interval, 
the vertex $\sigma_{2} \in \Delta_{2}$ does not satisfy the second condition in Definition \ref{def: almost delzant}. 
Thus, 
$V$ is not admissible with respect to the standard two-simplex $\Delta_{2}$.
\end{example}

\begin{example}
\label{example: subspace in R^4}
Let $V \subset \mathfrak{t}^{4}$ be a three-dimensional subspace generated by the following vectors:
\begin{equation*}
    \left[
        \begin{matrix}
            1 \\ 1 \\ 0 \\ 0
        \end{matrix}
    \right],
    \left[
        \begin{matrix}
            0 \\ 1 \\ 1 \\ 0
        \end{matrix}
    \right],
    \left[
        \begin{matrix}
            0 \\ 0 \\ 1 \\ 1
        \end{matrix}
    \right]
\end{equation*}
Then $V$ is not admissible with respect to the standard two-simplex $\Delta_{4}$ because $V$ fails to meet the first condition in Definition \ref{def: almost delzant}
Indeed, 
since the map $i_{V}^{\ast}:(\mathfrak{t}^{4})^{\ast} \to (\mathfrak{t}^{3})^{\ast}$ is given by
\begin{equation*}
    i_{V}^{\ast}(\xi_{1},\xi_{2},\xi_{3},\xi_{4}) = (\xi_{1}+\xi_{2},\xi_{2}+\xi_{3},\xi_{3}+\xi_{4}),
\end{equation*}
we see that 
\begin{align*}
    &i_{V}^{\ast}(v^{0}_{1}) = (1,0,0), \; 
    i_{V}^{\ast}(v^{0}_{2}) = (1,1,0), \\ 
    &i_{V}^{\ast}(v^{0}_{3}) = (0,1,1), \;
    i_{V}^{\ast}(v^{0}_{4}) = (0,0,1), \\
    &i_{V}^{\ast}(v^{1}_{1}) = (-1,0,0), \; 
    i_{V}^{\ast}(v^{1}_{2}) = (0,1,0), \\
    &i_{V}^{\ast}(v^{1}_{3}) = (-1,1,1), \;
    i_{V}^{\ast}(v^{1}_{4}) = (-1,0,1), \\
    &i_{V}^{\ast}(v^{2}_{1}) = (-1,-1,0), \; 
    i_{V}^{\ast}(v^{2}_{2}) = (0,-1,0), \\
    &i_{V}^{\ast}(v^{2}_{3}) = (-1,0,1), \;
    i_{V}^{\ast}(v^{2}_{4}) = (-1,-1,1), \\
    &i_{V}^{\ast}(v^{3}_{1}) = (0,-1,-1), \; 
    i_{V}^{\ast}(v^{3}_{2}) = (1,-1,-1), \\
    &i_{V}^{\ast}(v^{3}_{3}) = (1,0,-1), \;
    i_{V}^{\ast}(v^{3}_{4}) = (0,-1,0), \\
    &i_{V}^{\ast}(v^{4}_{1}) = (0,0,-1), \; 
    i_{V}^{\ast}(v^{4}_{2}) = (1,0,-1), \\
    &i_{V}^{\ast}(v^{4}_{3}) = (1,1,-1), \;
    i_{V}^{\ast}(v^{4}_{4}) = (0,1,0). 
\end{align*}
Hence, all the vertices $\sigma_{0},\ldots,\sigma_{4}$ are good with respect to the map $i_{V}^{\ast}$.
Since the polytope $i_{V}^{\ast}(\Delta_{4})$ (see Figure \ref{fig: cp4}) is the convex hull of the points
\begin{align*}
    &\sigma^{\prime}_{0} := i_{V}^{\ast}(\sigma_{0}) = (0,0,0), \;
    \sigma^{\prime}_{1} := i_{V}^{\ast}(\sigma_{1}) = (1,0,0), \\
    &\sigma^{\prime}_{2} := i_{V}^{\ast}(\sigma_{2}) = (1,1,0), \;
    \sigma^{\prime}_{3} := i_{V}^{\ast}(\sigma_{3}) = (0,1,1), \\
    &\sigma^{\prime}_{4} := i_{V}^{\ast}(\sigma_{4}) = (0,0,1),
\end{align*}
we see that the polytope $i_{V}^{\ast}(\Delta_{4})$ is not simple because the vertex $i_{V}^{\ast}(\sigma_{0})$ of $i_{V}^{\ast}(\Delta_{4})$ has four edges, 
i.e. the vertex $\sigma_{0}\in \Delta_{4}$ does not satisfy the first condition in Definition \ref{def: almost delzant}.
Thus, $V$ is not admissible with respect to the standard four-simplex $\Delta_{4}$.
\begin{figure}[h]
\begin{center}
    \includegraphics[keepaspectratio,width=7cm]{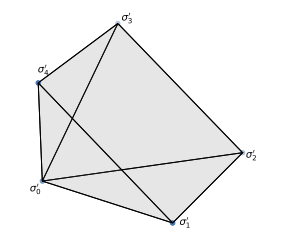}
    \caption{the polytope $i_{V}^{\ast}(\Delta_{4})$ in Example \ref{example: subspace in R^4}}
    \label{fig: cp4}
\end{center}   
\end{figure}
\end{example}



\bibliographystyle{alpha}
\bibliography{delzant-type}
\end{document}